\documentclass{amsart}

\usepackage{amssymb,amsmath}
\usepackage[section]{algorithm}
\usepackage{algpseudocode}

\newtheorem{theorem}{Theorem}[section]
\newtheorem{proposition}[theorem]{Proposition}

\newtheorem{corollary}[theorem]{Corollary}

\newtheorem{definition}[theorem]{Definition}

\newtheorem{remark}[theorem]{Remark}

\newtheorem{example}[theorem]{Example}

\def\N{{\mathbb N}}

\def\Q{{\mathbb Q}}

\def\lcm{{\mathrm{lcm}}}
\def\ord{{\mathrm{ord}}}

\def\w{{\mathrm{w}}}

\def\ord{{\mathrm{ord}}}
\def\topord{{\mathrm{topord}}}

\def\hSigma{{\hat{\Sigma}}}
\def\hN{{\hat{\N}}}

\def\bX{{\bar{X}}}

\def\bP{{\bar{P}}}

\def\bM{{\bar{M}}}

\def\End{{\mathrm{End}}}

\def\Mon{{\mathrm{Mon}}}

\def\Gr{Gr\"obner}

\def\lm{{\mathrm{lm}}}
\def\lc{{\mathrm{lc}}}
\def\lt{{\mathrm{lt}}}
\def\LM{{\mathrm{LM}}}
\def\spoly{{\mathrm{spoly}}}

\def\Reduce{{\textsc{Reduce}}}

\def\SigmaGBasis{{\textsc{SigmaGBasis}}}

\def\Inc{{\mathrm{Inc}}}

\begin{document}

\title[Gr\"obner bases and gradings for partial difference ideals]
{Gr\"obner bases and gradings for partial difference ideals}

\author[R. La Scala]{Roberto La Scala$^*$}

\address{$^*$ Dipartimento di Matematica, via Orabona 4, 70125 Bari, Italia}
\email{roberto.lascala@uniba.it}

\thanks{Partially supported by Universit\`a di Bari}



\maketitle

\begin{abstract}
In this paper we introduce a working generalization of the theory of
\Gr\ bases for algebras of partial difference polynomials with constant
coefficients. One obtains symbolic (formal) computation for systems
of linear or non-linear partial difference equations arising, for instance,
as discrete models or by the discretization of systems of differential equations.
From an algebraic viewpoint, the algebras of partial difference polynomials
are free objects in the category of commutative algebras endowed with
the action by endomorphisms of a monoid isomorphic to $\N^r$. Then,
the investigation of \Gr\ bases in this context contributes also to
the current research trend consisting in studying polynomial rings
under the action of suitable symmetries that are compatible with effective
methods. Since the algebras of difference polynomials are not Noetherian ones,
we propose in this paper a theory for grading them that provides a Noetherian
subalgebras filtration. This implies that the variants of the Buchberger's
algorithm we developed for difference ideals terminate in the finitely generated
graded case when truncated up to some degree. Moreover, even in the non-graded
case, we provide criterions for certifying completeness of eventually
finite \Gr\ bases when they are computed within sufficiently large
bounded degrees. We generalize also the concepts of homogenization and
saturation, and related algorithms, to the context of difference ideals.
The feasibily of the proposed methods is shown by an implementation
in Maple that is the first to provide computations for systems of non-linear
partial difference equations. We make use of a test set based on the
discretization of concrete systems of non-linear partial differential
equations.
\end{abstract}



\section{Introduction}

An important idea at the intersection of many algebraic theories consists
in studying algebraic structures under the action of operators of different
nature, typically automorphisms and derivations. Classical roots of this
idea can be found clearly in invariant and representation theory, as well as
in the study of polynomial identities satisfied by associative algebras.
Recently, topics like algebraic statistic \cite{BD,HM} or entanglement
theory \cite{MW} have given new impulse and applications to the research
on such themes. Another fundamental source of inspiration is the theory
of differential and difference algebras introduced in the pioneeristic
work of Ritt \cite{Ri1,Ri2} and afterwards developed by Kolchin \cite{Ko},
Cohn \cite{Co}, Levin \cite{Le} and many others. From the point of view
of computational methods, starting from the algorithms proposed by Ritt
himself, a considerable advancement can be recorded in the differential
case (see for instance \cite{Se}). Much less has been achieved for the
algebras of difference polynomials where working algorithms can be found
mainly in the linear case \cite{GR}. Nevertheless, the interest
for such computations is relevant because of applications in the discretization
of systems of differential equations like the automatic generation of finite
difference schemes or the consistency analysis of finite difference
approximations \cite{Ge, GBM, LM}. The present paper contributes to this
research trend by concerning the development of effective methods for
systems of linear or non-linear partial difference equations with constant
coefficients. We provide also an implementation of such methods which is
the first to allow computations in the non-linear case. Specifically,
we generalize the theory of \Gr\ bases and related algorithms for ideals
of the algebra of partial difference polynomials. We are able to do this
in a general and systematic way, by defining classes of suitable
monomial orderings, by extending the Buchberger's algorithm and the concept
of grading to difference ideals, by defining truncated homogeneous
computations and even by introducing a suitable notion of homogenization
for such ideals. First contributions to such theory can be found in
\cite{Ge,LSL,LSL2}. In particular, owing to the notion of ``letterplace
correspondence/embedding'' introduced in \cite{LS,LSL,LSL2}, note that
\Gr\ bases computations for ideals of the free associative algebra
are a subclass of the same computations for ideals of the algebra
of ordinary difference polynomials.

The algebras of partial difference polynomials are free algebras in the
class of commutative algebras that are invariant under the action
by endomorphisms of a monoid isomorphic to $\N^r$. Then, the study
of \Gr\ bases for such algebras belongs to the general investigation
of computational methods for commutative rings or modules that have
suitable symmetries. Moreover, from the viewpoint of applications,
the algebras of partial difference polynomials are fundamental
structures in the formal theory of partial difference equations where
a set of unknown multivariate functions is assumed algebraically independent
together with all partial shifts of them. To provide symbolic computation
for systems of such equations is hence essential to introduce \Gr\ bases
methods. Based on a suitable definition of monomial orderings that are
compatible with shifts action and the description of large classes of them,
the present paper introduces variants of the Buchberger's algorithm
for partial difference ideals. These procedures take advantage of the monoid
symmetry essentially by killing all S-polynomials in a orbit except
for a minimal one. Note that the algebras of difference polynomials are
not Noetherian since they are polynomial rings in an infinite number
of variables and hence termination is not generally guaranteed for the
proposed algorithms. With the aim of improving this situation, we define
suitable gradings that are compatible with the monoid action and provide
filtrations of the algebra of partial difference polynomials with finitely
generated subalgebras. We obtain therefore the termination for finitely
generated graded difference ideals when computations are performed
within some bounded degree. For non-graded ideals but for monomial orderings
compatible with such gradings, we prove also criterions able to certify
that a \Gr\ basis computation performed over a suitable finite set
of variables that is within a sufficiently large degree, is a complete one.
Finally, the paper generalizes the notion of homogenization and saturation
to difference ideals with respect to the given gradings and provides
the algorithms to perfom this ideal operations. As a byproduct, one obtains
an alternative algorithm to compute \Gr\ bases of non-graded difference ideals
via homogeneous computations. By means of an implementation in Maple,
all these methods are finally experimented on difference ideals obtained
by the discretization of systems of non-linear differential equations.


\section{Algebras of partial difference polynomials}

Fix $K$ any field and let $\Sigma$ be a monoid (semigroup with identity)
that we denote multiplicatively. Let $A$ be a commutative $K$-algebra and
denote $\End_K(A)$ the monoid of $K$-algebra endomorphisms of $A$.
We call $A$ a {\em $\Sigma$-invariant algebra} or briefly a
{\em $\Sigma$-algebra} if there is a monoid homomorphism
$\rho:\Sigma\to\End_K(A)$. In this case, we denote $\sigma\cdot x =
\rho(\sigma)(x)$, for all $\sigma\in\Sigma$ and $x\in A$. Let $A,B$ be
$\Sigma$-algebras and $\varphi:A\to B$ be a $K$-algebra homomorphism.
We say that $\varphi$ is a {\em $\Sigma$-algebra homomorphism} if
$\varphi(\sigma\cdot x) = \sigma\cdot\varphi(x)$, for all $\sigma\in\Sigma$
and $x\in A$. Let $A$ be a $\Sigma$-algebra and let $I\subset A$ be an ideal.
We call $I$ a {\em $\Sigma$-invariant ideal} or simply a {\em $\Sigma$-ideal}
if $\Sigma\cdot I\subset I$. Clearly, all kernels of $\Sigma$-algebra
homomorphisms are $\Sigma$-ideals.

\begin{definition}
Let $A$ be a $\Sigma$-algebra and let $X\subset A$ be a subset. We say that
$A$ is {\em $\Sigma$-generated by $X$} if $A$ is generated by $\Sigma\cdot X$
as $K$-algebra. In other words, $A$ coincides with the smallest
$\Sigma$-subalgebra of $A$ containing $X$. In the same way, one defines
$\Sigma$-generation for the $\Sigma$-ideals.
\end{definition}

In the category of $\Sigma$-invariant algebras one can define free objects.
In fact, let $X$ be a set and denote $x(\sigma)$ each element $(x,\sigma)$ of the
product set $X(\Sigma) = X\times\Sigma$. Define $P = K[X(\Sigma)]$
the polynomial algebra in the commuting variables $x(\sigma)$. For any
element $\sigma\in\Sigma$ consider the $K$-algebra endomorphism
$\bar{\sigma}:P\to P$ such that $x(\tau)\mapsto x(\sigma\tau)$,
for all $x(\tau)\in X(\Sigma)$. Then, one has a faithful monoid representation
$\rho:\Sigma\to\End_K(P)$ such that $\rho(\sigma) = \bar{\sigma}$
ad hence $P$ is a $\Sigma$-algebra. Note that if $\Sigma$ is a
left-cancellative monoid then all maps $\rho(\sigma)$ are injective. 

\begin{proposition}
Let $A$ be a $\Sigma$-algebra and let $f:X\to A$ be any map. Then, there is
a unique $\Sigma$-algebra homomorphism $\varphi:P\to A$ such that
$\varphi(x(1)) = f(x)$, for all $x\in X$.
\end{proposition}

\begin{proof}
It is sufficient to define $\varphi(x(\sigma)) = \sigma\cdot f(x)$,
for all $x\in X$ and $\sigma\in\Sigma$. In fact, one has
$\varphi(\tau\cdot x(\sigma)) = \varphi(x(\tau\sigma)) =
\tau\sigma\cdot f(x) = \tau\cdot(\sigma\cdot f(x)) = \tau\cdot
\varphi(x(\sigma))$, for any $\tau\in\Sigma$.
\end{proof}

\begin{definition}
We call $P = K[X(\Sigma)]$ the {\em free $\Sigma$-algebra generated by $X$}.
In fact, $P$ is $\Sigma$-generated by the subset $X(1) = \{x_i(1)\mid x_i\in X\}$.
\end{definition}

In other words, the algebra $P$ is an essential tool in the theory
of $\Sigma$-algebras because any such algebra $A$ that is $\Sigma$-generated
by a set $X$ can be obtained as a quotient $\Sigma$-algebra $P/I$, where
$I$ is a $\Sigma$-ideal of $P$. For instance, from the viewpoint of computational
methods, if one develops them for $P$ then such methods can be extended
to any quotient $P/I$ as it is done in the classical theory of \Gr\ bases
for affine algebras. Note also that if $\Sigma$ is defined as the monoid
$\Inc(\N) = \{f:\N\to\N\mid f\ \mbox{strictly increasing}\}$, or some power
of this, one obtains an environment for computations in algebraic statistic
\cite{BD,HM}.

We want now to go in the direction of developing fundamental structures
for symbolic (formal) computation on systems of partial difference equations
with constant coefficients. From now on, {\em we assume} that
$X = \{x_0,x_1,\ldots\}$ is a finite or countable set and $\Sigma$
is a free commutative monoid generated by a finite set, say
$\{\sigma_1,\ldots,\sigma_r\}$. Then, we consider the free $\Sigma$-algebra
$P = K[X(\Sigma)]$. Note that $(\Sigma,\cdot)$ is a cancellative monoid
isomorphic to $(\N^r,+)$ and the monomorphisms $\rho(\sigma):P\to P$ have
infinite order for all $\sigma\neq 1$. For any $x_i(\sigma)\in X(\Sigma)$,
we call $i$ and $\sigma$ respectively the {\em index} and the {\em weight}
of the variable $x_i(\sigma)$. If we put $X(\sigma) =
\{x_i(\sigma)\mid x_i\in X\}$ and $x_i(\Sigma) =
\{x_i(\sigma)\mid\sigma\in\Sigma\}$ one has clearly $P =
\bigotimes_{\sigma\in\Sigma} K[X(\sigma)] =
\bigotimes_{x_i\in X} K[x_i(\Sigma)]$, where all subalgebras $K[X(\sigma)]$
are isomorphic to $K[X]$ and all subalgebras $K[x_i(\Sigma)]$ to $K[\Sigma]$.

\begin{definition}
The free $\Sigma$-algebra $P = K[X(\Sigma)]$ $(\Sigma =
\langle \sigma_1,\ldots,\sigma_r \rangle)$ is called the {\em algebra
of partial difference polynomials with constant coefficients}.
\end{definition}

The motivation for such name is in the formal theory of partial difference
equations \cite{Co,Le}. In this theory, in fact, the indeterminates $x_i(1)$
are by definition algebraically independent unknown functions $u_i(t_1,\ldots,t_r)$
in the variables $t_j$ and the maps $\rho(\sigma_k)$ are the shift operators
$u_i(t_1,\ldots,t_r)\mapsto u_i(t_1,\ldots,t_k+h,\ldots,t_r)$
where $h$ is a parameter (mesh step). If $\sigma = \prod_i \sigma_i^{\alpha_i}$
then the indeterminates $x_i(\sigma) = \sigma\cdot x_i(1)$ correspond to
the (algebraically independent) shifted functions
$u_i(t_1 + \alpha_1 h,\ldots,t_r + \alpha_r h) =
\sigma\cdot u_i(t_1,\ldots,t_r)$. Then, a $\Sigma$-ideal $I\subset P$
is also called a {\em partial difference ideal} and a $\Sigma$-basis
of $I$ corresponds to a system of partial difference equations
in the unknown functions $u_i(t_1,\ldots,t_r)$. One uses the term
{\em ordinary difference} when $r = 1$.
Note that the algebras of difference polynomials are not Noetherian
rings since they are polynomial rings in an infinite number of variables.
One has therefore that difference ideals have bases or $\Sigma$-bases which
are generally infinite. 

In the next sections we generalize the \Gr\ basis theory to the free
$\Sigma$-algebra $P = K[X(\Sigma)]$ of partial difference polynomials.
Clearly, one reobtains the classical theory when $\Sigma = \{1\}$ ($r = 0$)
that is $P = K[X]$. The starting point is to define monomial orderings
of $P$ which are compatible with the action of the monoid $\Sigma$. 


\section{Monomial $\Sigma$-orderings}

Denote by $M = \Mon(P)$ the set of all monomials of $P$. Note that even
if the set $X(\Sigma)$ is infinite (in fact countable), one can endow $P$
by monomial orderings. This is an important consequence of the Higman's Lemma
\cite{Hi} which can be stated in the following way (see for instance \cite{AH},
Corollary 2.3 and remarks at beginning of page 5175).

\begin{proposition}
Let $\prec$ be a total ordering on $M$ such that
\begin{itemize}
\item[(i)] $1\preceq m$ for all $m\in M$;
\item[(ii)] $\prec$ is compatible with multiplication on $M$, that is
if $m\prec n$ then $t m\prec t n $, for any $m,n,t\in M$.
\end{itemize}
Then $\prec$ is also a well-ordering of $M$ that is a monomial
ordering of $P$ if and only if the restriction of $\prec$ to the variables
set $X(\Sigma)$ is a well-ordering.
\end{proposition}

Clearly, it is easy to construct well-orderings for the set $X(\Sigma)$ which is 
in bijective correspondence to $\N^{r+1}$. Note that the monoid $\Sigma$
stabilizes the monomials set $M$ since it stabilizes $X(\Sigma)$.
We introduce then the following notion.

\begin{definition}
Let $\prec$ be a monomial ordering of $P$. We call $\prec$ a {\em (monomial)
$\Sigma$-ordering of $P$} if $\prec$ is compatible with the $\Sigma$-action
on $M$, that is $m\prec n$ implies that $\sigma\cdot m\prec \sigma\cdot n$
for all $m,n\in M$ and $\sigma\in\Sigma$.
\end{definition}

A straightforward consequence of this definition is the following result.

\begin{proposition}
\label{sigmaincreases}
Let $\prec$ be a monomial $\Sigma$-ordering of $P$. Then $m\preceq\sigma\cdot m$
for all $m\in M$ and $\sigma\in\Sigma$.
\end{proposition}

\begin{proof}
By contradiction, assume that there are $m,\sigma$ such that $m\succ\sigma\cdot m$.
Then, $\sigma\cdot m\succ \sigma^2\cdot m$ and by induction one obtains
the infinite descending chain $m\succ\sigma\cdot m\succ\sigma^2\cdot m\succ\ldots$
which contradicts that $\prec$ is a well-ordering.
\end{proof}

The orderings on the variables set $X(\Sigma)$ that can be extended to
monomial $\Sigma$-orderings are as follows.

\begin{definition}
Let $\prec$ be a well-ordering of $X(\Sigma)$. We call $\prec$ a {\em (variable)
$\Sigma$-ranking of $P$} if $\prec$ is compatible with the $\Sigma$-action
on $X(\Sigma)$, that is $u\prec v$ implies that $\sigma\cdot u\prec \sigma\cdot v$
for all $u,v\in X(\Sigma)$ and $\sigma\in\Sigma$.
\end{definition}

As for Proposition \ref{sigmaincreases}, we have that if $\prec$ is a
$\Sigma$-ranking then $u\preceq \sigma\cdot u$ for all $u\in X(\Sigma)$
and $\sigma\in\Sigma$. Moreover, if $X$ is a finite set then condition
$u\preceq\sigma\cdot u$ for all $u,\sigma$ together with $\Sigma$-compatibility
implies that $\prec$ is a well-ordering by applying Dickson's Lemma
(or Higman's Lemma) to $\Sigma$ which is isomorphic to $\N^r$. However,
note that in this paper the set $X$ may be also countable.

Owing to the decompositions $X(\Sigma) = \bigcup_{\sigma\in\Sigma} X(\sigma) =
\bigcup_{x_i\in X} x_i(\Sigma)$ of the variables set of the ring $P$,
we can define $\Sigma$-rankings of $P$ in a natural way. Denote by $Q$
the monoid $K$-algebra defined by the free commutative monoid $\Sigma =
\langle \sigma_1,\ldots,\sigma_r \rangle$. In other words, $Q =
K[\sigma_1,\ldots,\sigma_r]$ is the polynomial algebra
in the commutative variables $\sigma_i$. From now on, we assume that
$\Sigma$ is endowed with a monomial ordering $<$ of $Q$. By abuse,
we call $<$ a {\em monomial ordering of $\Sigma$}.

\begin{definition}
Fix $<$ a monomial ordering of $\Sigma$. For all
$x_i(\sigma),x_j(\tau)\in X(\Sigma)$, we define:
\begin{itemize}
\item[(i)] $x_i(\sigma)\prec x_j(\tau)$ if and only if
$\sigma < \tau$ or $\sigma = \tau$ and $i < j$. In other words,
$X(\sigma)\prec X(\tau)$ when $\sigma < \tau$.
\item[(ii)] $x_i(\sigma) \prec' x_j(\tau)$ if and only if
$i < j$ or $i = j$ and $\sigma < \tau$. In other words,
$x_i(\Sigma)\prec' x_j(\Sigma)$ when $i < j$.
\end{itemize}
Clearly $\prec$ and $\prec'$ are both $\Sigma$-rankings of $P$ that
we call respectively {\em weight} and {\em index $\Sigma$-ranking
defined by a monomial ordering of $\Sigma$}.
\end{definition}

For all $x_i\in X$ and $\sigma\in\Sigma$ denote $P(\sigma) = K[X(\sigma)],
M(\sigma) = \Mon(P(\sigma))$ and $P(x_i) = K[x_i(\Sigma)],
M(x_i) = \Mon(P(x_i))$. Owing to the tensor decompositions $P =
\bigotimes_{\sigma\in\Sigma} P(\sigma) = \bigotimes_{x_i\in X} P(x_i)$,
one has that a monomial $m\in M$ can be uniquely written as
$m = m(\delta_1)\cdots m(\delta_k) = m(x_{i_1})\cdots m(x_{i_l})$,
where $m(\delta_p)\in M(\delta_p), m(x_{i_p})\in M(x_{i_p})$ and
$\delta_1 > \ldots > \delta_k, i_1 > \ldots > i_l$. By means of such
presentations we can define block monomial orderings of $P$
extending weight and index ranking. Recall that $\rho:\Sigma\to \End_K(P)$
is the faithful monoid representation defined by the action of
$\Sigma$ over $P$. For any $\sigma\in\Sigma$ one has that the map
$\rho(\sigma)$ defines an isomorphism between the monoids $M(1),M(\sigma)$
and hence between the algebras $P(1),P(\sigma)$. In other words, we have
$M(\sigma) = \sigma\cdot M(1),P(\sigma) = \sigma\cdot P(1)$.

\begin{definition}
\label{weightmord}
Fix $\prec$ a monomial ordering of the subalgebra $P(1)\subset P$
and extend it to all subalgebras $P(\sigma)$ $(\sigma\in\Sigma)$ by the
isomorphisms $\rho(\sigma)$. In other words, we put
$\sigma\cdot m\prec \sigma\cdot n$ if and only if $m\prec n$,
for any $m,n\in M(1)$. Then, for all $m,n\in M,
m = m(\delta_1)\cdots m(\delta_k), n = n(\delta_1)\cdots n(\delta_k)$
with $\delta_1 > \ldots > \delta_k$ we define $m\prec_w n$ if and only if
$m(\delta_j) = n(\delta_j)$ if $j < i$ and $m(\delta_i)\prec n(\delta_i)$
for some $1\leq i\leq k$. Clearly, the restriction of $\prec_w$ to the
variables of $P$ is just the weight $\Sigma$-ranking.
\end{definition}

\begin{proposition}
The ordering $\prec_w$ is a $\Sigma$-ordering of $P$.
\end{proposition}

\begin{proof}
Note that if $m = m(\delta_1)\cdots m(\delta_k)\in M$ with
$m(\sigma_i)\in M(\sigma_i)$ and $\delta_1 > \ldots > \delta_k$ then
$\sigma\cdot m = m(\sigma\delta_1)\cdots m(\sigma\delta_k)$, where
$m(\sigma\delta_i) = \sigma\cdot m(\delta_i)\in M(\sigma\delta_i)$
and $\sigma\delta_1 > \ldots > \sigma\delta_k$ since $<$ is a monomial
ordering of $\Sigma$. Assume $m\prec_w n$ that is $m(\delta_j) = n(\delta_j)$
for $j < i$ and $m(\delta_i)\prec n(\delta_i)$. Clearly $m(\sigma\delta_j) =
n(\sigma\delta_j)$ for $j < i$ and one has $m(\delta_i)\prec n(\delta_i)$
if and only if $m(1)\prec n(1)$ if and only if
$m(\sigma\delta_i)\prec n(\sigma\delta_i)$. Then, we conclude that
$\sigma\cdot m\prec_w \sigma\cdot n$.
\end{proof}

Note that we have also a monoid faithful representation $\phi:\N \to \End_K(P)$
such that the endomorphism $\phi(i)$ is defined as
$x_j(\sigma)\mapsto x_{i+j}(\sigma)$ for any $i,j\geq 0$ and $\sigma\in\Sigma$.
Clearly $\phi(i)$ induces isomorphism between the monoids $M(x_0),M(x_i)$
and the algebras $P(x_0),P(x_i)$. The algebra $P(x_0)$ can be easily endowed
with a $\Sigma$-ordering. For instance, since $P(x_0) =
\bigotimes_{\sigma\in\Sigma} K[x_0(\sigma)]$ one can define a lexicographic
ordering as in Definition \ref{weightmord}.

\begin{definition}
\label{indexmord}
Fix $\prec$ a monomial $\Sigma$-ordering of the subalgebra $P(x_0)\subset P$
and extend it to all subalgebras $P(x_i)$ $(x_i\in X)$ by the isomorphisms
$\phi(i)$. For any $m,n\in M, m = m(x_{i_1})\cdots m(x_{i_k}),
n = n(x_{i_1})\cdots n(x_{i_k})$ with $i_1 > \ldots > i_k$ we put
$m\prec_i n$ if and only if $m(x_{i_q}) = n(x_{i_q})$ if $q < p$ and
$m(x_{i_p})\prec n(x_{i_p})$ for some $1\leq p\leq k$. Note that
the restriction of $\prec_i$ to the variables of $P$ is the index
$\Sigma$-ranking.
\end{definition}

\begin{proposition}
The ordering $\prec_i$ is a $\Sigma$-ordering of $P$.
\end{proposition}

\begin{proof}
Note that if $m = m(x_{i_1})\cdots m(x_{i_k})\in M$ with
$m(x_{i_p})\in M(x_{i_p})$ and $i_1 > \ldots > i_k$ then
$\sigma\cdot m = m'(x_{i_1})\cdots m'(x_{i_k})$ where
$m'(x_{i_p}) = \sigma\cdot m(x_{i_p})\in M(x_{i_p})$.
Suppose $m\prec_i n$ that is $m(x_{i_q}) = n(x_{i_q})$ if $q < p$
and $m(x_{i_p})\prec n(x_{i_p})$. We have clearly that
$m'(x_{i_q}) = n'(x_{i_q})$. Moreover, since $\prec$ is a
$\Sigma$-ordering of $P(x_0)$ and therefore of $P(x_{i_p})$,
one has also $m'(x_{i_p})\prec n'(x_{i_p})$ that is
$\sigma\cdot m\prec_i \sigma\cdot n$.
\end{proof}

We call the above monomial $\Sigma$-orderings $\prec_w,\prec_i$
of $P$ respectively {\em weight $\Sigma$-ordering defined by a monomial
ordering of $P(1)$} and {\em index $\Sigma$-ordering of $P$ defined by a
monomial $\Sigma$-ordering of $P(x_0)$}. Clearly, both these orderings
depend also on a monomial ordering of $\Sigma$. Note that index
$\Sigma$-orderings are suitable for generation of finite difference schemes
for partial differential equations \cite{GB,GBM}. The weight
$\Sigma$-orderings are instead compatible with the gradings
of the $\Sigma$-algebra $P$ we introduce in Section 6. For this reason
they are suitable for obtaining complete \Gr\ bases from truncated
computations (see Proposition \ref{finsigmacrit}).

To make things more explicit, we give now an example of a weight
and an index $\Sigma$-ordering.  Fix $X = \{x,y,z\}$ and $\Sigma =
\langle \sigma_1,\sigma_2 \rangle$. To simplify the notation,
we identify the monoid $(\Sigma,\cdot)$ with $(\N^2,+)$ by means
of the isomorphism $\sigma_1^i\sigma_2^j \mapsto (i,j)$.
Then, we fix the {\em degrevlex} monomial ordering on $\Sigma$
with $\sigma_1 > \sigma_2$ that is
\[
\ldots > (2,0) > (1,1) > (0,2) > (1,0) > (0,1) > (0,0)
\]
and assume $P(x) = K[x(i,j)\mid i,j\geq 0]$ be endowed with the {\em lex}
monomial ordering such that
\[
\ldots\succ x(2,0)\succ x(1,1)\succ x(0,2)\succ x(1,0)\succ x(0,1)\succ x(0,0).
\]
Finally, we fix also the {\em lex} ordering on $P(0,0) =
K[x(0,0),y(0,0),z(0,0)]$ with $x(0,0)\succ y(0,0)\succ z(0,0)$.
By isomorphisms, one has clearly the same ordering on $P(y),P(z)$
and $P(i,j) = K[x(i,j),y(i,j),z(i,j)]$, for all $i,j\geq 0, (i,j)\neq (0,0)$.
Then, a weight $\Sigma$-ordering is defined on $P = K[x(i,j),y(i,j),z(i,j)\mid
i,j\geq 0]$ as the block monomial ordering corresponding to the tensor
decomposition
\[
P = \ldots\otimes P(2,0)\otimes P(1,1)\otimes P(0,2)\otimes
P(1,0)\otimes P(0,1)\otimes P(0,0).
\]
In a similar way, one defines an index $\Sigma$-ordering on $P$
owing to the decomposition
\[
P = P(x)\otimes P(y)\otimes P(z).
\]
Similar $\Sigma$-orderings have been used for the examples contained
in Section 5 and 8 and for the computational experiments presented
in Section 9 (see also the Appendix).


\section{\Gr\ $\Sigma$-bases}

From now on, we consider $P$ endowed with a monomial $\Sigma$-ordering $\prec$.
Let $f = \sum_i c_i m_i\in P$ with $m_i\in M, c_i\in K,c_i\neq 0$.
We denote as usual $\lm(f) = m_k = \max_\prec\{m_i\}$, $\lc(f) = c_k$
and $\lt(f) = \lc(f)\lm(f)$.
If $G\subset P$ we put $\lm(G) = \{\lm(f) \mid f\in G,f\neq 0\}$ and
we denote as $\LM(G)$ the ideal of $P$ generated by $\lm(G)$. 

\begin{proposition}
Let $G\subset P$. Then $\lm(\Sigma\cdot G) = \Sigma\cdot \lm(G)$.
In particular, if $I$ is a $\Sigma$-ideal of $P$ then $\LM(I)$
is also a $\Sigma$-ideal.
\end{proposition}

\begin{proof}
Since $P$ is endowed with a $\Sigma$-ordering, one has that
$\lm(\sigma\cdot f) = \sigma\cdot\lm(f)$ for any $f\in P,f\neq 0$ and
$\sigma\in\Sigma$. Then, $\Sigma\cdot\lm(I) = \lm(\Sigma\cdot I)\subset
\lm(I)$ and therefore $\LM(I) = \langle \lm(I) \rangle$ is a 
$\Sigma$-ideal.
\end{proof}

\begin{definition}
Let $I\subset P$ be a $\Sigma$-ideal and $G\subset I$.
We call $G$ a {\em \Gr\ $\Sigma$-basis} of $I$ if $\lm(G)$ is a
$\Sigma$-basis of $\LM(I)$. In other words, $\Sigma\cdot G$
is a \Gr\ basis of $I$ as $P$-ideal.
\end{definition}

Since the monoid $\Sigma$ is assumed isomorphic to $\N^r$ that is
$\Sigma$-ideals are partial difference ideals, we may say that \Gr\
$\Sigma$-bases are {\em partial difference \Gr\ bases} \cite{Ge}.
Another possible name is {\em $\Sigma$-equivariant \Gr\ bases} \cite{BD}.
Simplicity and generality lead us to the previous definition
that already appeared in \cite{LSL2}.

Let $f,g\in P,f,g\neq 0$ and put $\lt(f) = c m, \lt(g) = d n$
with $m,n\in M$ and $c,d\in K$. If $l = \lcm(m,n)$ we define
as usual the {\em S-polynomial} $\spoly(f,g) = (l/c m) f - (l/d n) g$.
Clearly $\spoly(f,g) = - \spoly(g,f)$ and $\spoly(f,f) = 0$.

\begin{proposition}
\label{sigmaspoly}
For all $f,g\in P,f,g\neq 0$ and for any $\sigma\in\Sigma$ one has
$\sigma\cdot\spoly(f,g) = \spoly(\sigma\cdot f,\sigma\cdot g)$.
\end{proposition}

\begin{proof}
Since $\Sigma$ acts on the variables set $X(\Sigma)$ by injective maps,
it is sufficient to note that $\sigma\cdot \lcm(m,n) =
\lcm(\sigma\cdot m,\sigma\cdot n)$ for all $m,n\in M$ and $\sigma\in\Sigma$.
\end{proof}

The following definition is a standard tool in \Gr\ bases theory.

\begin{definition}
Let $f\in P,f\neq 0$ and $G\subset P$. If $f = \sum_i f_i g_i$ with
$f_i\in P,g_i\in G$ and $\lm(f)\succeq\lm(f_i)\lm(g_i)$ for all $i$,
we say that {\em $f$ has a \Gr\ representation respect to $G$}.
\end{definition}
 
Note that if $f = \sum_i f_i g_i$ is a \Gr\ representation then
$\sigma\cdot f = \sum_i (\sigma\cdot f_i)(\sigma\cdot g_i)$
is also a \Gr\ representation, for any $\sigma\in\Sigma$.
In fact, since $\prec$ is a $\Sigma$-ordering of $P$ one has
that $\lm(f)\succeq\lm(f_i)\lm(g_i)$ implies that
$\lm(\sigma\cdot f) = \sigma\cdot \lm(f)\succeq
(\sigma\cdot \lm(f_i))(\sigma\cdot \lm(g_i)) = \lm(\sigma\cdot f_i)
\lm(\sigma\cdot g_i)$ for all $i$.
A celebrated result from Bruno Buchberger \cite{Bu} is the following.

\begin{proposition}[Buchberger's criterion]
\label{buchcrit}
Let $G$ be a basis of the ideal $I\subset P$. Then, $G$ is a \Gr\ basis
of $I$ if and only if for all $f,g\in G,f,g\neq 0$ the S-polynomial
$\spoly(f,g)$ has a \Gr\ representation with respect to $G$.
\end{proposition}

Usually the above result, see for instance \cite{Ei}, is stated when $P$
is a polynomial algebra with a finite number of variables and $G$ is a
finite set. In fact, such assumptions are not needed since Noetherianity
is not used in the proof, but only the existence of a monomial ordering
for $P$. See also the comprehensive Bergman's paper \cite{Be} where
the ``Diamond Lemma'' is proved without any restriction on the finiteness
of the variables set. We want now to prove a generalization of the
Buchberger's criterion for \Gr\ $\Sigma$-bases of $P$. For this purpose
it is useful to introduce the following notations.

\begin{definition}
Let $\sigma = \prod_i \sigma_i^{\alpha_i}, \tau = \prod_i \sigma_i^{\beta_i}
\in \Sigma$. We denote $\gcd(\sigma,\tau) = \prod_i \sigma_i^{\gamma_i}$
where $\gamma_i = \min(\alpha_i,\beta_i)$, for any $i$.
\end{definition}

\begin{proposition}[$\Sigma$-criterion]
\label{sigmacrit}
Let $G$ be a $\Sigma$-basis of a $\Sigma$-ideal $I\subset P$.
Then, $G$ is a \Gr\ $\Sigma$-basis of $I$ if and only if for all
$f,g\in G,f,g\neq 0$ and for any $\sigma,\tau\in\Sigma$ such that
$\gcd(\sigma,\tau) = 1$, the S-polynomial
$\spoly(\sigma\cdot f, \tau\cdot g)$ has a \Gr\ representation
with respect to $\Sigma\cdot G$.
\end{proposition}

\begin{proof}
We prove that $\Sigma\cdot G$ is a \Gr\ basis of $I$ and we make use
of the Proposition \ref{buchcrit}. Then, consider any pair of elements
$\sigma\cdot f, \tau\cdot g\in \Sigma\cdot G$ where $f,g\in G,f,g\neq 0$ and
$\sigma,\tau\in\Sigma$. Put $\delta = \gcd(\sigma,\tau)$ and hence
$\sigma = \delta \sigma', \tau = \delta \tau'$ with $\sigma',\tau'\in\Sigma,
\gcd(\sigma',\tau') = 1$. By Proposition \ref{sigmaspoly} we have
$\spoly(\sigma\cdot f, \tau\cdot g) =
\delta\cdot \spoly(\sigma'\cdot f, \tau'\cdot g)$. By hypothesis, assume that
$\spoly(\sigma'\cdot f, \tau'\cdot g) = h =
\sum_\nu f_\nu (\nu\cdot g_\nu)$, with $\nu\in\Sigma,f_\nu\in P,g_\nu\in G$,
is a \Gr\ representation with respect to $\Sigma\cdot G$.
Since $\prec$ is a $\Sigma$-ordering of $P$, we conclude that we have also
the \Gr\ representation $\spoly(\sigma\cdot f, \tau\cdot g) =
\delta\cdot h = \sum_\nu (\delta\cdot f_\nu) (\delta\nu\cdot g_\nu)$.
\end{proof}

For the purpose of obtaining an effective Buchberger's algorithm from
the above criterion, note that all usual criteria (product criterion,
chain criterion, etc) can be used also in such procedure.
In particular, the arguments contained in the proof of Proposition
\ref{finsigmacrit} (see comments after this proof) imply that for
any pair of elements $f,g\in G$ and for all $\sigma,\tau\in\Sigma$
there are only a finite number of S-polynomials
$\spoly(\sigma\cdot f,\tau\cdot g)$ satisfying both the criteria
$\gcd(\sigma,\tau) = 1$ and $\gcd(\sigma\cdot\lm(f),\tau\cdot\lm(g))\neq 1$.

A standard subroutine in the Buchberger's algorithm is the following.

\suppressfloats[b]
\begin{algorithm}\caption{\Reduce}
\begin{algorithmic}[0]
\State \text{Input:} $G\subset P$ and $f\in P$.
\State \text{Output:} $h\in P$ such that $f - h\in\langle G\rangle$
and $h = 0$ or $\lm(h)\notin\LM(G)$.
\State $h:= f$;
\While{ $h\neq 0$ and $\lm(h)\in\LM(G)$ }
\State choose $g\in G,g\neq 0$ such that $\lm(g)$ divides $\lm(h)$;
\State $h:= h - (\lt(h)/\lt(g)) g$;
\EndWhile;
\State \Return $h$.
\end{algorithmic}
\end{algorithm}

\newpage
Note that the termination of $\Reduce$ is provided since $\prec$
is a monomial ordering of $P$. In particular, even if $G$ is an infinite set,
there are only a finite number of elements $g\in G,g\neq 0$ such that
$\lm(g)$ divides $\lm(h)$ and hence $\lm(g)\preceq\lm(h)$.
It is well-known that if $\Reduce(f,G) = 0$ then $f$ has a \Gr\ representation
with respect to $G$. Moreover, if $\Reduce(f,G) = h\neq 0$ then clearly
one has $\Reduce(f,G\cup\{h\}) = 0$. Therefore, from Proposition \ref{sigmacrit}
and product criterion it follows immediately the correctness of the following
algorithm.

\suppressfloats[b]
\begin{algorithm}\caption{SigmaGBasis}
\begin{algorithmic}[0]
\State \text{Input:} $H$, a $\Sigma$-basis of a $\Sigma$-ideal
$I\subset P$.
\State \text{Output:} $G$, a \Gr\ $\Sigma$-basis of $I$.
\State $G:= H$;
\State $B:= \{(f,g) \mid f,g\in G\}$;
\While{$B\neq\emptyset$}
\State choose $(f,g)\in B$;
\State $B:= B\setminus \{(f,g)\}$;
\ForAll{$\sigma,\tau\in\Sigma$ such that $\gcd(\sigma,\tau) = 1,
\gcd(\sigma\cdot\lm(f),\tau\cdot\lm(g))\neq 1$}
\State $h:= \Reduce(\spoly(\sigma\cdot f,\tau\cdot g), \Sigma\cdot G)$;
\If{$h\neq 0$}
\State $B:= B\cup\{(g,h),(h,h) \mid g\in G\}$;
\State $G:= G\cup\{h\}$;
\EndIf;
\EndFor;
\EndWhile;
\State \Return $G$.
\end{algorithmic}
\end{algorithm}

\newpage
Note that the above algorithm can be viewed as a variant of the usual
Buchberger's procedure applied for the basis $\Sigma\cdot H$, where an
additional criterion to avoid ``useless pairs'' is given by Proposition
\ref{sigmacrit}. Unfortunately, owing to Non-Noetherianity of the ring $P$,
the termination of \SigmaGBasis\ is not provided in general and this is,
in fact, one of the main problems in differential/difference algebra.
Precisely, even if a $\Sigma$-ideal $I\subset P$ has a finite
$\Sigma$-basis this may be not the case for the initial $\Sigma$-ideal
$\LM(I)$ that is all \Gr\ $\Sigma$-bases of $I$ are infinite sets.
Despite this bad general case, in Section 6 we introduce suitable
gradings for the algebra $P$ which provides that truncated versions
of the algorithm \SigmaGBasis\ with homogeneous input stops
in a finite number of steps. Note finally that some variant of \SigmaGBasis\
appeared in \cite{Ge} and before in \cite{LSL,LSL2} for the ordinary
difference case. In fact, the notion of ``letterplace correspondence/embedding''
introduced in these latter papers strictly relates non-commutative \Gr\ bases
to \Gr\ $\Sigma$-bases of ordinary difference ideals (see also \cite{LS}).


\section{An illustrative example}

In this section we apply the algorithm \SigmaGBasis\ to a simple example
in order to provide a concrete computation with it. Let $X = \{x,y\}, \Sigma =
\langle \sigma_1,\sigma_2 \rangle$ and consider the algebra of partial
difference polynomials $P = K[X(\Sigma)]$. To simplify the notation,
we identify the monoid $(\Sigma,\cdot)$ with $(\N^2,+)$ by means of the
isomorphism $\sigma_1^i\sigma_2^j \mapsto (i,j)$. Then, we denote the variables
of $P$ as $x(i,j),y(i,j)$, for all $i,j\geq 0$. We consider now
the $\Sigma$-ideal (difference ideal) $I\subset P$ that is $\Sigma$-generated
by the difference polynomials
\[
\begin{array}{l}
g_1 = y(1,1)y(1,0) - 2x(0,1)^2, \\
g_2 = y(2,0) + x(0,0)x(1,0).
\end{array}
\]
In other words, this $\Sigma$-basis (difference basis) encodes a system
of non-linear difference equations with constant coefficients
in two unknown bivariate functions.
By symbolic (formal) computations, we want to substitute this system
with a completion of it, namely a \Gr\ $\Sigma$-basis. We may want to do
this for the purposes of checking membership of other equations to the
$\Sigma$-ideal, elimination of unknowns, etc. The main problem is that such basis
may be infinite, but it is not the case for this example. We fix then
the {\em degrevlex} ordering on $\Sigma$ with $\sigma_1 > \sigma_2$ that is
on $\N^2$ where $(1,0) > (0,1)$. Moreover, we consider the {\em lex} monomial
ordering on $K[x(0,0),y(0,0)]$ with $x(0,0)\succ y(0,0)$. A weight $\Sigma$-ordering
(Definition \ref{weightmord}) is hence defined on $P =
\bigotimes_{(i,j)\in\N^2} K[x(i,j),y(i,j)]$ as a block monomial ordering.
In practice, it is the lexicographic monomial ordering based on the following
weight $\Sigma$-ranking
\[
\begin{array}{l}
\ldots \succ x(2, 0)\succ y(2, 0)\succ x(1, 1)\succ y(1, 1)\succ
x(0, 2)\succ y(0, 2)\succ x(1, 0)\succ \\
y(1, 0)\succ x(0, 1)\succ y(0, 1)\succ x(0, 0)\succ y(0, 0).
\end{array}
\]
We use this monomial $\Sigma$-ordering of $P$ for computing a \Gr\ $\Sigma$-basis
of $I$. Such basis consists of the elements $g_1,g_2$ together with
the difference polynomials
\[
\begin{array}{l}
g_3 = y(1,2)x(0,1)^2 - y(1,0)x(0,2)^2, \\
g_4 = 2x(1,1)^2 - x(0,0)x(1,0)x(0,1)x(1,1).
\end{array}
\]
Note that in all these elements the first monomial is the leading one
with respect to the given ordering of $P$. Let us see how the algorithm
\SigmaGBasis\ is able to obtain such \Gr\ $\Sigma$-basis.
Since the $\Sigma$-ideal $I$ is $\Sigma$-generated by $G = \{g_1,g_2\}$
then $I$ is generated as an ideal of $P$ by $\Sigma\cdot G$ that are
the polynomials
\[
\begin{array}{l}
(i,j)\cdot g_1 = y(i+1,j+1)y(i+1,j) - 2x(i,j+1)^2, \\
(i,j)\cdot g_2 = y(i+2,j) + x(i,j)x(i+1,j), \\
\end{array}
\]
for all $i,j\geq 0$. By applying the product criterion, we have to consider
three kinds of S-polynomials
\[
\begin{array}{l}
\spoly((i,j+1)\cdot g_1,(i,j)\cdot g_1) =
(i,j)\cdot \spoly((0,1)\cdot g_1, g_1), \\

\spoly((i+1,j)\cdot g_1, (i,j)\cdot g_2) =
(i,j)\cdot \spoly((1,0)\cdot g_1, g_2), \\

\spoly((i+1,j+1)\cdot g_1, (i,j+2)\cdot g_2) =
(i,j+1)\cdot \spoly((1,0)\cdot g_1, (0,1)\cdot g_2). \\
\end{array}
\]
The $\Sigma$-criterion implies therefore that one has to reduce
with respect to the basis $\Sigma\cdot G$ just the S-polynomials
\[
\begin{array}{l}
s_1 = \spoly((0,1)\cdot g_1, g_1), \\
s_2 = \spoly((1,0)\cdot g_1, g_2), \\
s_3 = \spoly((1,0)\cdot g_1, (0,1)\cdot g_2).
\end{array}
\]
The reduction of the S-polynomial $s_1$ leads to the new element $g_3$
and the current $\Sigma$-basis of $I$ is now $G = \{g_1,g_2,g_3\}$.
The additional S-polynomials that survive to product and $\Sigma$-criterion
are
\[
\begin{array}{l}
s_4 = \spoly((0,1)\cdot g_1, g_3), \\
s_5 = \spoly((0,2)\cdot g_1, g_3), \\
s_6 = \spoly((0,2)\cdot g_2, (1,0)\cdot g_3).
\end{array}
\]
We have that $s_4\to 0$ and $s_2\to g_4$ with respect to $\Sigma\cdot G$.
The $\Sigma$-basis is then $G = \{g_1,g_2,g_3,g_4\}$ and one has a new S-polynomial
\[
s_7 = \spoly((1,0)\cdot g_3, g_4).
\]
Finally, we have that all S-polynomials $s_3,s_5,s_6,s_7$ reduce to zero
with respect to $\Sigma\cdot G$ and hence $G$ is a \Gr\ $\Sigma$-basis of $I$.
Note that $G$ is in fact a minimal such basis and also that in this simple example
there is no use of the chain criterion that can be always applied together
with the other criteria.


\section{Gradings of $P$ compatible with $\Sigma$-action}

We want now to introduce some gradings of the algebra $P = K[X(\Sigma)]$
which are compatible with $\Sigma$-action and formation of least common
multiples in $M = \Mon(P)$. As before, we fix a monomial order $<$ of
$\Sigma$. We start extending the structure $(\Sigma,\max,\cdot)$ in the
following way.

\begin{definition}
Let $0$ be an element disjoint with $\Sigma$ and put $\hSigma =
\Sigma\cup \{0\}$. Then, we define a commutative idempotent monoid
$(\hSigma,+)$ with identity $0$ that extends the monoid $(\Sigma,\max)$
(with identity 1) by imposing that $0 + \sigma = \sigma$, for any
$\sigma\in\hSigma$. Moreover, we define a commutative monoid $(\hSigma,\cdot)$
with identity 1 extending the monoid $(\Sigma,\cdot)$ by putting
$0\cdot\sigma = 0$, for all $\sigma\in\hSigma$. Since multiplication clearly
distributes over addition, one has that $(\hSigma,+,\cdot)$ is a commutative
idempotent semiring, also known as commutative dioid \cite{GM}.
\end{definition}

Note that the faithful monoid representation $\rho:\Sigma\to\End_K(P)$
can be extended to $\hSigma$ where $\rho(0):P\to P$ is the algebra
endomorphism such that $x_i(\sigma)\mapsto 0$, for all $x_i(\sigma)\in X(\Sigma)$.

\begin{definition}
Let $\w:M\to\hSigma$ be the unique mapping such that
\begin{itemize}
\item[(i)] $\w(1) = 0$;
\item[(ii)] $\w(m n) = \w(m) + \w(n)$, for any $m,n\in M$;
\item[(iii)] $\w(x_i(\sigma)) = \sigma$, for all $i\geq 0$ and $\sigma\in\Sigma$.
\end{itemize}
Note that (i),(ii) state that $\w$ is a monoid homomorphism from
the free commutative monoid $(M,\cdot)$ to $(\hSigma,+)$. We call $\w$
the {\em weight function} of $P$.
\end{definition}

More explicitely, if
$m = x_{i_1}(\delta_1)^{\alpha_1}\cdots x_{i_k}(\delta_k)^{\alpha_k}$
is any monomial of $P$ different from 1 then $\w(m) =
\delta_1 + \cdots + \delta_k = \max_<(\delta_1,\ldots,\delta_k)$. We denote
$M_\sigma = \{m\in M\mid \w(m) = \sigma\}$ and define $P_\sigma\subset P$
the subspace spanned by $M_\sigma$, for any $\sigma\in\hSigma$.
Because $\w:(M,\cdot)\to (\hSigma,+)$ is a monoid homomorphism one has that
$P = \bigoplus_{\sigma\in\hSigma} P_\sigma$ is a grading of the algebra $P$
over the commutative monoid $(\hSigma,+)$. If $f\in P_\sigma$ we say that
$f$ is a {\em $\w$-homogeneous element} and we put $\w(f) = \sigma$.
Recall that for any $\sigma\in\Sigma$ we denoted $P(\sigma) = K[X(\sigma)]$
which is a subalgebra of $P = K[X(\Sigma)]$ isomorphic to $K[X]$. If we put
$P(0) = P_0 = K$ then one has that $P^{(\sigma)} =
\bigoplus_{\tau\leq\sigma} P_\sigma = \bigotimes_{\tau\leq\sigma} P(\tau)$
is a subalgebra of $P$. In particular, we have that $P^{(1)} = P_0 \oplus P_1
= P(0) \otimes P(1) = P(1)$ is isomorphic to the polynomial algebra $K[X]$.

\begin{definition}
A monomial order $<$ of $\Sigma$ is said to be {\em sequential} if
$\{\tau\in\Sigma\mid\tau\leq\sigma\}$ is a finite set, for all
$\sigma\in\Sigma$.
\end{definition}

It is important to note that if $X$ is a finite set and $<$ is a sequential
ordering of $\Sigma$ then the sequence $\{P^{(\sigma)}\mid \sigma\in\hSigma\}$
is a filtration of $P$ consisting of Noetherian subalgebras. For such reason,
from now on {\em we assume} $\Sigma$ be endowed with a sequential monomial
ordering.

\begin{proposition}
\label{wgood}
The weight function satisfies the following properties:
\begin{itemize}
\item[(i)] $\w(\sigma\cdot m) = \sigma \w(m)$, for any $\sigma\in\Sigma$
and $m\in M$;
\item[(ii)] $\w(\lcm(m,n)) = \w(m n) = \w(m) + \w(n)$, for all $m,n\in M$.
Then, $m\mid n$ implies that $\w(m)\leq \w(n)$.
\end{itemize}
\end{proposition}

\begin{proof}
If $m = 1$ then $\w(\sigma\cdot m) = \w(m) = 0 = \sigma \w(m)$.
If otherwise $m = x_{i_1}(\delta_1)^{\alpha_1}\cdots x_{i_k}(\delta_k)^{\alpha_k}$
with $\delta_1 > \ldots > \delta_k$ then $\sigma\cdot m =
x_{i_1}(\sigma\delta_1)^{\alpha_1}\cdots x_{i_k}(\sigma\delta_k)^{\alpha_k}$
where $\sigma\delta_1 > \ldots > \sigma\delta_k$ since $<$ is a monomial
ordering of $\Sigma$. We conclude that $\w(\sigma\cdot m) = \sigma\delta_1 =
\sigma\w(m)$. To prove (ii) it is sufficient to note that the weight of a
monomial does not depend on the exponents of the variables occuring in it.
\end{proof}

Note that the property (i) implies that the map $\w$ is a homomorphism
with respect to the action of $\Sigma$ on $M$ and $\hSigma$. In other words,
one has that $\sigma P_\tau\subset P_{\sigma\tau}$ for any $\sigma\in\Sigma,
\tau\in\hSigma$. Moreover, the property (ii) means that $\w$ is also a monoid
homomorphism from $(M,\lcm)$ to $(\hSigma,+)$.

\begin{definition}
Let $I$ be an ideal of $P$. We call $I$ a {\em $\w$-graded ideal} if $I =
\sum_\sigma I_\sigma$ with $I_\sigma = I\cap P_\sigma$. In this case, if
$I$ is also a $\Sigma$-ideal then $\sigma\cdot I_\tau\subset I_{\sigma\tau}$
for all $\sigma\in\Sigma,\tau\in\hSigma$.
\end{definition}

Owing to the $\w$-grading of $P$, one can show that a truncated
version of the algorithm \SigmaGBasis\ admits termination.
If $f,g\in P,f\neq g$ are $\w$-homogeneous elements then the S-polynomial
$h = \spoly(f,g)$ is clearly $\w$-homogeneous too. Moreover, by property (ii)
of Proposition \ref{wgood}, we have that $\w(h) = \w(f) + \w(g)$
and hence if $\w(f),\w(g)\leq \delta$ then also $\w(h)\leq \delta$,
for some $\delta\in\Sigma$. By means of this remark, one obtains
immediately the following result.

\begin{proposition}[Truncated termination over the weight]
\label{wtermin}
Let $I\subset P$ be a $\w$-graded $\Sigma$-ideal and fix $\delta\in\Sigma$.
Assume $I$ has a $\w$-homogeneous basis $H$ such that $H_\delta =
\{f\in H\mid \w(f)\leq\delta\}$ is a finite set. Then, there is
a $\w$-homogeneous \Gr\ $\Sigma$-basis $G$ of $I$ such that $G_\delta$
is also a finite set. In other words, if we consider for the algorithm
\SigmaGBasis\ a selection strategy of the S-polynomials based on their
weights ordered by $<$, we obtain that the $\delta$-truncated
version of \SigmaGBasis\ stops in a finite number of steps.
\end{proposition}

\begin{proof}
First of all, note that the algorithm \SigmaGBasis\ computes essentially
a subset $G$ of a \Gr\ basis $\Sigma\cdot G$ obtained by applying the
Buchberger's algorithm to the basis $\Sigma\cdot H$ of $I$. Moreover,
by Proposition \ref{wgood} the elements of $\Sigma\cdot H$ and hence
of $\Sigma\cdot G$ are all $\w$-homogeneous.
Denote $H'_\delta = \{\sigma\cdot f\mid\sigma\in\Sigma,f\in H,
\sigma\w(f)\leq \delta\}$. Since $<$ is a sequential monomial order of $\Sigma$
and $H_\delta$ is a finite set one has that $H'_\delta$ is also
a finite set. We consider therefore $X_\delta$ the finite set of variables
of $P$ occurring in the elements of $H'_\delta$ and define
$P_{(\delta)} = K[X_\delta]\subset P$. In fact, the $\delta$-truncated
algorithm \SigmaGBasis\ computes a subset of a \Gr\ basis of the ideal
$I_{(\delta)} \subset P_{(\delta)}$ generated by $H'_\delta$.
By Noetherianity of the finitely generated polynomial ring $P_{(\delta)}$
we clearly obtain termination.
\end{proof}

Clearly the above result provides algorithmic solution to the ideal
membership problem for finitely generated $\w$-graded $\Sigma$-ideals.
Note that if $r = 0$ that is $\Sigma = \{1\}$ then the algorithm \SigmaGBasis\
coincides with classical Buchberger's algorithm and Proposition \ref{wtermin}
states that if $I$ is a finitely generated ideal of $P = P_0\oplus P_1 =
K[x_0,x_1,\ldots]$ then $I$ has also a finite \Gr\ basis. According with
the above proof, this is a consequence of the fact that the Buchberger's
algorithm runs over the finite number of variables occuring in the generators
of $I$.

Another useful grading of $P$ can be introduced in the following way.
Consider the set $\hN = \N\cup\{-\infty\}$ endowed with the binary operations
$\max$ and $+$. Then $(\hN,\max,+)$ is clearly a commutative idempotent
semiring (or commutative dioid or max-plus algebra).
Define $\deg:\hSigma\to\hN$ the mapping such that $\deg(0) = -\infty$
and $\deg(\sigma) = \sum_i \alpha_i$, for any $\sigma =
\prod_i \sigma_i^{\alpha_i}$. Clearly $\deg$ is a monoid homomorphism
from $(\hSigma,\cdot)$ to $(\hN,+)$.

\begin{definition}
Let $\ord:M\to\hN$ be the unique mapping such that
\begin{itemize}
\item[(i)] $\ord(1) = -\infty$;
\item[(ii)] $\ord(m n) = \max(\ord(m),\ord(n))$, for any $m,n\in M$;
\item[(iii)] $\ord(x_i(\sigma)) = \deg(\sigma)$, for all $i\geq 0$
and $\sigma\in\Sigma$.
\end{itemize}
Clearly (i),(ii) state that $\ord$ is a monoid homomorphism from
$(M,\cdot)$ to $(\hN,\max)$. We call $\ord$ the {\em order function} of $P$.
\end{definition}

For any monomial $m =
x_{i_1}(\delta_1)^{\alpha_1}\cdots x_{i_k}(\delta_k)^{\alpha_k}$ different
from 1 we have that $\ord(m) = \max(\deg(\delta_1),\ldots,\deg(\delta_k))$.
Clearly, the order function defines a grading $P = \bigoplus_{d\in\hN} P_d$
of the algebra $P$ over the commutative monoid $(\hN,\max)$. Define
$P^{(d)} = \bigoplus_{i\leq d} P_i = \bigotimes_{\deg(\sigma)\leq d} P(\sigma)$
which is a subalgebra of $P$. Then, if $X$ is a finite set we have that
the sequence $\{ P^{(d)}\mid d\in\hN \}$ is a filtration of $P$
with Noetherian subalgebras where $P^{(0)} = P_{-\infty}\oplus P_0$ is
isomorphic to $K[X]$. 

\begin{definition}
A monomial order $<$ of $\Sigma$ is said to be {\em compatible with $\deg$}
when $\deg(\sigma) < \deg(\tau)$ implies that $\sigma < \tau$,
for any $\sigma,\tau\in\Sigma$.
\end{definition}

If $<$ is compatible with $\deg$, note that $<$ is a sequential ordering
of $\Sigma$ and $\ord(m) = \deg(\w(m))$, for all $m\in M$. Finally, one has
that the weight and order functions clearly coincide when $r = 1$.

\begin{proposition}
\label{ordgood}
The order function satisfies the following:
\begin{itemize}
\item[(i)] $\ord(\sigma\cdot m) = \deg(\sigma) + \ord(m)$, for any
$\sigma\in\Sigma$ and $m\in M$;
\item[(ii)] $\ord(\lcm(m,n)) = \ord(m n) = \max(\ord(m),\ord(n))$,
for all $m,n\in M$. Therefore, if $m\mid n$ then
$\ord(m)\leq \ord(n)$.
\end{itemize}
\end{proposition}

\begin{proof}
For $m = 1$ one has $\ord(\sigma\cdot m) = \ord(m) = -\infty =
\deg(\sigma) + \ord(m)$. If otherwise $m =
x_{i_1}(\delta_1)^{\alpha_1}\cdots x_{i_k}(\delta_k)^{\alpha_k}$
then $\sigma\cdot m =
x_{i_1}(\sigma\delta_1)^{\alpha_1}\cdots x_{i_k}(\sigma\delta_k)^{\alpha_k}$
and hence $\ord(\sigma\cdot m) =
\max(\deg(\sigma\delta_1),\ldots,\deg(\sigma\delta_k)) =
\deg(\sigma) + \max(\deg(\delta_1),\ldots,\deg(\delta_k)) =
\deg(\sigma) + \ord(m)$.
Property (ii) follows immediately as in Proposition \ref{wgood}.
\end{proof}

\begin{definition}
Let $I$ be an ideal of $P$. We call $I$ a {\em $\ord$-graded ideal} if
$I = \sum_i I_i$ with $I_i = I\cap P_i$. If $I$ is also a $\Sigma$-ideal
then $\sigma\cdot I_i\subset I_{\deg(\sigma) + i}$
for any $\sigma\in\Sigma$ and $i\in\hN$.
\end{definition}

Consider now $f,g\in P,f\neq g$ two $\ord$-homogeneous elements.
The S-polynomial $h = \spoly(f,g)$ is clearly $\ord$-homogeneous
and $\ord(h) = \max(\ord(f),\ord(g))$. Then $\ord(f),\ord(g)\leq d$
implies that $\ord(h)\leq d$, for some $d\in\N$ and one proves the
following result as for Proposition \ref{wtermin}.

\begin{proposition}[Truncated termination over the order]
\label{ordtermin}
Let $I\subset P$ be a $\ord$-graded $\Sigma$-ideal and fix $d\in\N$.
Assume $I$ has a $\ord$-homogeneous basis of $H$ such that $H_d =
\{f\in H\mid \ord(f)\leq d\}$ is a finite set. Then, there is
a $\ord$-homogeneous \Gr\ $\Sigma$-basis $G\subset I$ such that $G_d$
is also a finite set. In other words, if we consider for \SigmaGBasis\
a selection strategy of the S-polynomials based on their orders, we have
that the $d$-truncated version of \SigmaGBasis\ terminates in a finite
number of steps.
\end{proposition}

By means of weight and order functions one has criterions, also in the
non-graded case, that provide that a \Gr\ $\Sigma$-basis is the eventually
finite complete one even if it has been computed within some bounded
weight or order for the algebra $P$ that is over a finite number of variables.
This is of course important because actual computations can be only performed
in such a way. As before, we fix a sequential monomial ordering $<$ on $\Sigma$.

\begin{definition}
Let $\prec$ be a monomial $\Sigma$-ordering of $P$. We call $\prec$
{\em compatible with the weight function} if $\w(m) < \w(n)$ implies that
$m\prec n$, for all $m,n\in M$. In a similar way, one defines when $\prec$
is {\em compatible with the order function}.
\end{definition}

\begin{proposition}
Let $\prec_w$ be a weight $\Sigma$-ordering as in Definition \ref{weightmord}.
Then $\prec_w$ is compatible with the weight function. In particular,
if the monomial order $<$ of $\Sigma$ is compatible with $\deg$ then
$\prec_w$ is also compatible with the order function.
\end{proposition}

\begin{proof}
Let $m = m(\delta_1)\cdots m(\delta_k),n = n(\delta_1)\cdots n(\delta_k)$
two monomials of $P$ with $m(\delta_i),n(\delta_i)\in M(\delta_i)$ and
$\delta_1 > \ldots > \delta_k$. Assume $m\prec_w n$ that is
$m(\delta_j) = n(\delta_j)$ if $j < i$ and $m(\delta_i)\prec n(\delta_i)$
for some $1\leq i\leq k$. If $i > 1$ or $m(\delta_i)\neq 1$ then clearly 
$\w(m) = \w(n) = \delta_1$. Otherwise, we conclude $\w(m) < \delta_1 = \w(n)$.
Moreover, if $<$ is compatible with $\deg$ then $\ord(m) = \deg(\w(m)) <
\deg(\w(n)) = \ord(n)$ implies that $\w(m) < \w(n)$ and hence $m\prec_w n$.
\end{proof}

For any $\delta\in\Sigma, d\in\N$ define now $\Sigma_\delta =
\{\sigma\in\Sigma\mid\sigma\leq\delta\}$ and $\Sigma_d =
\{\sigma\in\Sigma\mid\deg(\sigma)\leq d\}$. 

\begin{proposition}[Finite $\Sigma$-criterion]
\label{finsigmacrit}
Assume the $\Sigma$-ordering of $P$ is compatible with the weight function.
Let $G\subset P$ be a finite set and denote $I$ the $\Sigma$-ideal
generated by $G$. Moreover, define $\delta = \max_<\{\w(\lm(g))\mid g\in G\}$.
Then, $G$ is a \Gr\ $\Sigma$-basis of $I$ if and only if for all $f,g\in G$
and for any $\sigma,\tau\in\Sigma$ such that $\gcd(\sigma,\tau) = 1$ and
$\gcd(\sigma\cdot\lm(f),\tau\cdot\lm(g))\neq 1$, the S-polynomial
$\spoly(\sigma\cdot f,\tau\cdot g)$ has a \Gr\ representation
with respect to the finite set $\Sigma_{\delta^2}\cdot G$.
In the same way, if the $\Sigma$-ordering of $P$ is compatible with the
order function and $d = \max\{\ord(\lm(g))\mid g\in G\}$, then
$G$ is a \Gr\ $\Sigma$-basis of $I$ when the above S-polynomials have
a \Gr\ representation with respect to $\Sigma_{2d}\cdot G$.
\end{proposition}

\begin{proof}
Let $\spoly(\sigma\cdot f,\tau\cdot g) = h = \sum_\nu f_\nu (\nu\cdot g_\nu)$
be a \Gr\ representation with respect to $\Sigma\cdot G$ that is
$\lm(h)\succeq \lm(f_\nu)(\nu\cdot \lm(g_\nu))$ for all $\nu$. We want to
bound the elements $\nu\in\Sigma$ with respect to the ordering $<$.
Put $m = \lm(f),n = \lm(g)$ and hence $\lm(\sigma\cdot f) =
\sigma\cdot m,\lm(\sigma\cdot g) = \sigma\cdot n$. By product criterion,
we can assume that $u = \gcd(\sigma\cdot m,\tau\cdot n)\neq 1$.
Then, there is a variable $x_i(\sigma \alpha) = x_i(\tau \beta)$
that divides $u$ where $x_i(\alpha)$ divides $m$ and hence
$\alpha\leq\w(m)\leq \delta$ and $x_i(\beta)$ divides $n$
and therefore $\beta\leq\w(n)\leq \delta$. Then $\sigma \alpha = \tau \beta$
and one has that $\sigma\mid \beta,\tau\mid \alpha$ because
$\gcd(\sigma,\tau) = 1$. We conclude that $\sigma,\tau\leq\delta$ and
if $v = \lcm(\sigma\cdot m,\tau\cdot m)$ then $\w(v) =
\max(\sigma\w(m),\tau\w(n))\leq \delta^2$.
Clearly $v\succ \lm(h)\succeq \nu\cdot\lm(g_\nu)$ and hence $\delta^2\geq\w(v)
\geq \nu \w(\lm(g_\nu))\geq \nu$. In a similar way, one argues for the
order function.
\end{proof}

The above criterion implies that with respect to $\Sigma$-orderings compatible
with weight or order functions one has an algorithm able to compute
a finite \Gr\ $\Sigma$-basis, whenever this exists, in a finite number of steps.
To fix ideas, let us consider only weights. If $G$ is a finite
$\Sigma$-basis of $I$ and $\delta = \max_<\{\w(\lm(g))\mid g\in G\}$,
we may start considering all S-polynomials $\spoly(\sigma\cdot f,\tau\cdot g)$
with $f,g\in G$ and $\sigma,\tau\in\Sigma$ such that $\gcd(\sigma,\tau) = 1$ and
$\gcd(\sigma\cdot\lm(f),\tau\cdot\lm(g))\neq 1$. Note that such S-polynomials
are in a finite number since in the above proof we observed that
$\sigma,\tau\leq\delta$ and the monomial ordering of $\Sigma$ is sequential.
Moreover, one has also that these S-polynomials can be reduced only
by elements of the finite set $\Sigma_{\delta^2}\cdot G$.
If, as a result of some reduction, a new element $f\neq 0$ has to be added
to the $\Sigma$-basis $G$ and $\w(\lm(f)) = \delta' > \delta$
then it is sufficient to update the weight bound $\delta$ to $\delta'$.


\section{Homogenizing with respect to order function}

The purpose of this section is to analyze (de)homogenization processes
in the context of $\Sigma$-ideals. Such methods are generally developed
to have structures and computations that are homogeneous with respect
to some grading, even if the input data are not such. Besides to
the theoretical advantages as the concept of projective closure,
these techniques usually imply computational benefits (see for instance
\cite{BCR}). Note that for $\Sigma$-ideals it is completely useless
to consider classical gradings (total degree, multidegree, etc)
since they not provide compatibility conditions with the $\Sigma$-action
like the ones contained in Proposition 6.4 and Proposition 6.9.
We decided then to present (de)homogenization methods only for the
grading defined by the order function because univariate homogenizations
are usually more efficient than multivariate ones since leading monomials
are preserved by the homogenization process. Note finally that
a major difference of the theory we present here with the classical
one is that the kernel of the dehomogenizing homomorphism contains
a non-trivial graded ideal which implies that the homogenization
process has to be considered modulo such ideal.

Let $t$ be a new variable disjoint with $X$. Define $\bX = X\cup\{t\},
\bX(\Sigma) = \bX \times \Sigma, \bP = K[\bX(\Sigma)]$ and finally
$\bM = \Mon(\bP)$. Consider the algebra endomorphism $\varphi:\bP\to\bP$
such that $x_i(\sigma)\mapsto x_i(\sigma)$ and $t(\sigma)\mapsto 1$,
for all $i,\sigma$. Clearly $\varphi^2 = \varphi$ and $P = \varphi(\bP)$.
Moreover, one has that $\varphi$ is a $\Sigma$-algebra endomorphism.
Then $\varphi$ defines a bijective correspondence between all $\Sigma$-ideals
of $P$ and $\Sigma$-ideals of $\bP$ containing $\ker\varphi =
\langle t(1) - 1 \rangle_\Sigma$.

\begin{definition}
Denote by $N = N_\ord$ the largest $\ord$-graded $\Sigma$-ideal contained
in $\ker\varphi$ that is the ideal generated by all $\ord$-homogeneous
elements $f\in\bP$ such that $\varphi(f) = 0$.
\end{definition}

\begin{proposition}
\label{hidbas}
The ideal $N\subset\bP$ is generated by the elements
\begin{itemize}
\item[(i)] $t(\sigma) - t(\tau)$ for all $\sigma,\tau\in\Sigma,\sigma\neq\tau,
\deg(\sigma) = \deg(\tau)$;
\item[(ii)] $t(\sigma) t(\tau) - t(\sigma),x(\sigma) t(\tau) - x(\sigma)$
for any $\sigma,\tau\in\Sigma,\deg(\sigma)\geq\deg(\tau)$.
\end{itemize}
\end{proposition}

\begin{proof}
Let $f\in\bP$ be a $\ord$-homogeneous element such that $\varphi(f) = 0$.
Since the polynomials of type (i),(ii) clearly belongs to $N$, we have
to prove that $f$ is congruent to $0$ modulo (i),(ii).
Assume first that all variables of $f$ belong to $t(\Sigma) =
\{t(\sigma)\mid \sigma\in\Sigma\}$. Recall that if $m =
t(\delta_1)^{\alpha_1}\cdots t(\delta_k)^{\alpha_k}$ is any monomial
of $f$ then $d = \ord(f) = \max(\deg(\delta_1),\ldots,\deg(\delta_k))$.
Therefore, one has that $f$ is congruent modulo (ii) to $f' =
\sum_i c_i t(\tau_i)$ where $\tau_i\in\Sigma,\deg(\tau_i) = d$
and $c_i\in K, \sum_i c_i = 0$. By applying identity (i) it follows that
$f$ is congruent to $(\sum_i c_i) t(\sigma) = 0$ for some fixed $\sigma$
such that $\deg(\sigma) = d$.

Consider now the general case when the variables of $f$ belong to
$\bX(\Sigma)$. Fix $\sigma\in\Sigma$ such that $\deg(\sigma) = d$.
Modulo the identities (i),(ii), one has that $f$ is congruent to a
polynomial $f'$ whose monomials are either of type $m\in M$ such that
$\ord(m) = d$ or of type $t(\sigma) n$ where $n\in M, \ord(n) < d$.
We show that in fact $f' = 0$. Denote $f' = t(\sigma) g - h$ where
$g,h$ are polynomials in $P$, $h$ is $\ord$-homogeneous and $\ord(h) = d$.
Since $0 = \varphi(f') = g - h$ one has that $f' = (t(\sigma) - 1) g$.
If we assume $g\neq 0$ then the monomials $n$ of $g$ are such that
$\ord(n) = d$ which is a contradiction. 
\end{proof}

We want now to define a bijective correspondence between all $\Sigma$-ideals
of $P$ and some class of $\ord$-graded $\Sigma$-ideals of $\bP$ containing $N$.

\begin{definition}
Let $I$ be any $\Sigma$-ideal of $P$. We define $I^*\subset\bP$ the largest
$\ord$-graded $\Sigma$-ideal contained in the preimage $\varphi^{-1}(I)$
that is $I^*$ is the ideal generated by all $\ord$-homogeneous elements
in $\varphi^{-1}(I)$. Clearly $N = 0^*\subset I^*$. We call $I^*$ the
{\em $\ord$-homogenization} of the $\Sigma$-ideal $I$.
\end{definition}

\begin{definition}
\label{hpoly}
Let $f\in P, f\neq 0$ and denote $f = \sum_d f_d$ the decomposition
of $f$ in its $\ord$-homogeneous components. We define $\topord(f) =
d' = \max\{d\}$. If $f\in K$ that is $d' = -\infty$ we put $f^* = f$.
Otherwise, we denote $f^* = t(\sigma) f$ where $\sigma\in\Sigma$
such that $\deg(\sigma) = d'$. We call $\topord(f)$ the {\em top order} of $f$
and $f^*$ its {\em $\ord$-homogenization}.
\end{definition}

Note that $\varphi(f^*) = f$ and hence the element $f^*$ is essentially
defined modulo the ideal $N$ (see also the next result). Owing to
generators (i) of $N$ in Proposition \ref{hidbas}, all variables
$t(\sigma)$ such that $\deg(\sigma) = d'$ are congruent modulo $N$. We don't
need then to specify which of these variables we use for defining
$f^* = t(\sigma) f$.

\begin{proposition}
Let $I$ be a $\Sigma$-ideal of $P$. Then $I^* = \langle f^*\mid
f\in I, f\neq 0\rangle + N$.
\end{proposition}

\begin{proof}
Denote $J = \langle f^*\mid f\in I\rangle + N$. Clearly $J\subset I^*$.
Let $g\in I^*$ be a $\ord$-homogeneous element and define $f = \varphi(g)\in I$.
If $f = 0$ then $g\in N\subset J$. Otherwise, denote $d = \topord(f)$ and
$d' = \ord(g)$. Since clearly $d'\geq d$ one has that $g$ is congruent
modulo $N$ to $h = t(\sigma) f$, where $\sigma\in\Sigma$ such that
$\deg(\sigma) = d'$. Hence, if $d' = d$ then $h$ is congruent exactly to $f^*$.
Otherwise, the polynomial $h$ is congruent to $t(\sigma) f^*$.
In both cases, we conclude that $g$ is congruent modulo $N\subset J$
to an element of $J$ and therefore $g\in J$.
\end{proof}

If $I\subset P$ is a $\Sigma$-ideal one has clearly that $\varphi(I^*) = I$.
Moreover, if $J\subset\bP$ is a $\ord$-graded $\Sigma$-ideal containing $N$
then in general $J\subset \varphi(J)^*$.

\begin{definition}
Let $N\subset J\subset\bP$ be a $\ord$-graded $\Sigma$-ideal.
Define $J' = \varphi(J)^* = \langle \varphi(f)^* \mid f\in J,f\notin N,
f\ \ord\mbox{-homogeneous} \rangle + N$. Then $J\subset J'\subset\bP$ is
a $\ord$-graded $\Sigma$-ideal that we call the {\em saturation} of $J$.
\end{definition}

\begin{definition}
Let $J\subset\bP$ be a $\ord$-graded $\Sigma$-ideal containing $N$.
We say that $J$ is {\em saturated} if $J$ coincides with its saturation
$\varphi(J)^*$ that is for any $\ord$-homogeneous element $f\in J,f\notin N$
one has that $\varphi(f)^*\in J$. If $I$ is a $\Sigma$-ideal of $P$ then its
$\ord$-homogenization $I^*$ is clearly a saturated ideal.
\end{definition}

Therefore, a bijective correspondence is given between all $\Sigma$-ideals of $P$
and the saturated $\ord$-graded $\Sigma$-ideals of $\bP$ containing $N$. 

We want now to analyze the behaviour of \Gr\ $\Sigma$-bases under homogenization
and dehomogenization. Note that the arguments of Proposition \ref{hidbas}
implies clearly that the polynomials (i),(ii) are in fact a \Gr\ basis
of the ideal $N$ with respect to any monomial ordering of $\bP$. 
For this reason we introduce the following notion.

\begin{definition}
\label{Nnform}
A monomial $m\in \bM$ is said to be {\em normal modulo $N$} if $m\in M$
or $m = t(\sigma) n$ with $n\in M,\sigma\in\Sigma$ such that
$d = \deg(\sigma) > \ord(n)$. Moreover, we require that $t(\sigma) =
\min_\prec\{t(\tau)\mid \deg(\tau) = d\}$. A polynomial $f\in\bP$
is {\em in normal form modulo $N$} if all its monomials are normal
modulo $N$.
\end{definition}

Note that owing to generators (i) of the ideal $N$, we choose
$t(\sigma) = \min_\prec\{t(\tau)\mid \deg(\tau) = d\}$
since in this case $\lm(t(\sigma) - t(\tau)) = t(\tau)$.

\begin{definition}
\label{homorder}
Let $\prec$ be a $\Sigma$-ordering of $\bP$ compatible with the order
function. We call $\prec$ a {\em $\ord$-homogenization $\Sigma$-ordering}
if $t(\sigma) m\prec n$ for all $m,n\in M,\sigma\in\Sigma$ such that
$\deg(\sigma) = \ord(n) > \ord(m)$.
\end{definition}

It is easy to define one of the above orderings. Fix for instance
the {\em lex} or {\em degrevlex} monomial order on the polynomial
ring $\bP(1) = K[x_0(1),x_1(1),\ldots,t(1)]$ where
$x_0(1)\succ x_1(1)\succ \ldots\succ t(1)$. Moreover, fix a monomial ordering
on $\Sigma$ which is compatible with $\deg$ and define the weight
$\Sigma$-ordering $\prec_w$ of $\bP$ as in Definition \ref{weightmord}. 
Clearly $\prec_w$ is a $\ord$-homogenization $\Sigma$-ordering.

From now on, {\em we assume} $\bP$ be endowed with a $\ord$-homogenization
$\Sigma$-ordering.

\begin{proposition}
\label{phicompat}
Let $p,q\in\bM$ be two normal monomials modulo $N$ such that $\ord(p) = \ord(q)$.
Then $p\prec q$ implies that $\varphi(p)\prec\varphi(q)$.
\end{proposition}

\begin{proof}
By definition, the monomials $p,q$ are of type $m\in M$ or $t(\sigma) m$
with $\deg(\sigma) > \ord(m)$. Since $\prec$ is a $\ord$-homogenization order,
when comparing two of such monomials of the same order one has only
the following cases: $m\prec n$, $t(\sigma) m\prec t(\sigma) n$ or
$t(\sigma) m\prec n$. Then, we have to prove $\varphi(p) = m\prec n =
\varphi(q)$ only when $t(\sigma) m\prec n$. This follows immediately
from $\prec$ is compatible with the order function and $\ord(m) < \ord(n) =
\deg(\sigma)$.
\end{proof}

From now on, for any $f\in P,f\neq 0$ we denote by $f^*$ the normal form
of $t(\sigma)f$ modulo $N$ where $\sigma\in\Sigma,\deg(\sigma) = \topord(f)$.

\begin{proposition}
\label{phicompat2}
Let $f\in\bP,f\neq 0$ be a $\ord$-homogeneous polynomial in normal form
modulo $N$. Then $\lm(\varphi(f)) = \varphi(\lm(f))$. Moreover, we have
that $\lm(f^*) = \lm(f)$ for all $f\in P,f\neq 0$.
\end{proposition}

\begin{proof}
The first part of the statement follows immediately from Proposition
\ref{phicompat}. Moreover, if $\sigma\in\Sigma,\deg(\sigma) = \topord(f)$
then $t(\sigma)$ cannot appear in the leading monomial of $f^*$ and hence
$\lm(f^*) = \lm(f)$.
\end{proof}

\begin{definition}
Let $N\subset J\subset \bP$ be a $\Sigma$-ideal. Moreover, let $G\subset J$
be a subset of polynomials in normal form modulo $N$. We say that $G$ is a
{\em \Gr\ $\Sigma$-basis of $J$ modulo $N$} if $G\cup N$ is a \Gr\
$\Sigma$-basis of $J$.
\end{definition}

\begin{proposition}
\label{dehomgb}
Let $N\subset J\subset \bP$ be a $\ord$-graded $\Sigma$-ideal. If $G$ is
a $\ord$-homogeneous \Gr\ $\Sigma$-basis of $J$ modulo $N$ then $\varphi(G)$
is a \Gr\ $\Sigma$-basis of $\varphi(J)$.
\end{proposition} 

\begin{proof}
Since $G$ is a \Gr\ $\Sigma$-basis of $J$ modulo $N$ we have that 
for any $\ord$-homogeneous polynomial $f\in J,f\neq 0$ in normal form
modulo $N$ there is an element $g\in G$ and $\sigma\in\Sigma$ such that
$\sigma\cdot \lm(g) \mid \lm(f)$. Then, by applying the $\Sigma$-algebra
endomorphism $\varphi$ one obtains that $\sigma\cdot \lm(\varphi(g))\mid
\lm(\varphi(f))$ that is $\varphi(G)$ is a \Gr\ $\Sigma$-basis of $\varphi(J)$.
\end{proof}

\begin{proposition}
\label{homgb}
Let $I\subset P$ be a $\Sigma$-ideal and let $G$ be a \Gr\ $\Sigma$-basis of $I$. 
Then $G^* = \{g^*\mid g\in G\}$ is a $\ord$-homogeneous \Gr\ basis of $I^*$
modulo $N$. Moreover, one has that $\lm(G^*) = \lm(G)$.
\end{proposition}

\begin{proof}
Let $f'\in I^*$ be a $\ord$-homogeneous element in normal form modulo $N$
and put $f = \varphi(f')\in I$. Then, either $f' = f^*$ or $f' = t(\sigma) f^*$
with $\ord(f') = \deg(\sigma) > \ord(f^*) = \topord(f)$. Since $G$ is a
\Gr\ $\Sigma$-basis of $I$ there is $g\in G,\tau\in\Sigma$ such that
$\tau\cdot\lm(g)\mid \lm(f)$. By Proposition \ref{phicompat2} one has that
$\lm(f) = \lm(f^*)$ and $\lm(g) = \lm(g^*)$. Therefore, $\tau\cdot\lm(g^*)$
divides $\lm(f^*)$ and this monomial clearly divides $\lm(f')$.
\end{proof}

By the above propositions we obtain immediately what follows.

\begin{corollary}
\label{satgb}
Let $N\subset J\subset\bP$ be a $\ord$-graded $\Sigma$-ideal and denote
$J' = \varphi(J)^*$ its saturation.  Moreover, let $G$ be a $\ord$-homogeneous
\Gr\ $\Sigma$-basis of $J$ modulo $N$. Then $G' = \varphi(G)^* =
\{ \varphi(g)^* \mid g\in G \}$ is a $\ord$-homogeneous \Gr\ $\Sigma$-basis
of $J'$ modulo $N$. Moreover, we have $\lm(G') = \lm(\varphi(G))$.
\end{corollary}

Let $I\subset P$ be any $\Sigma$-ideal. The previous results suggest
an alternative method to calculate a \Gr\ $\Sigma$-basis of $I$ which is
based only on $\ord$-homogeneous computations. Assume $H$ is any $\Sigma$-basis
of $I$ and denote as before $H^* = \{ f^*\mid f\in H \}$. Clearly $J =
\langle H^* \rangle_\Sigma + N$ is a $\ord$-homogeneous $\Sigma$-ideal
of $\bP$ containing $N$ such that $\varphi(J) = I$. Assume now
we compute $G$ a $\ord$-homogeneous \Gr\ $\Sigma$-basis of $J$ modulo $N$.
Then, $\varphi(G)$ is a \Gr\ $\Sigma$-basis of $I$.
Note that by using a $\ord$-based selection strategy for the S-polynomials,
the \Gr\ $\Sigma$-basis $G$ can be obtained order by order automatically
{\em minimal} that is $\sigma\cdot\lm(f)$ not divides $\lm(g)$ for all $f,g\in G,
f\neq g$ and $\sigma\in\Sigma$. This is clearly a computational advantage,
but since generally $\lm(G)\neq\lm(\varphi(G))$ one has that $\varphi(G)$
may be not minimal. In the worst case, the ideal $J$ may have an infinite
and hence uncomputable minimal \Gr\ $\Sigma$-bases but $I$ has a finite one.
This is clearly not the case when one considers a saturated ideal $J' = I^*$
since we have $\lm(G') = \lm(\varphi(G'))$ when $G'$ is a minimal
\Gr\ $\Sigma$-basis of $J'$. Note that this nice property depends on the fact
that we deal with a univariate homogenization.
A drawback is that if one computes the saturation $J'$ by means of
the ideal $J$ according to Corollary \ref{satgb}, one has again to compute
a \Gr\ $\Sigma$-basis of $J$. Then, a better approach consists in computing
``on the fly'' the \Gr\ $\Sigma$-basis of $J'$ starting from the generating
set $\{f^* \mid f\in H\}$. In other words, any time that a new generator $g$
of the $\ord$-homogeneous \Gr\ $\Sigma$-basis arises from the reduction
of an S-polynomial, we saturate $g$ that is we substitute this polynomial
with $\varphi(g)^*$. In formal terms, the algorithm one obtains is
the following one.

\suppressfloats[b]
\begin{algorithm}\caption{SigmaGBasis2}
\begin{algorithmic}[0]
\State \text{Input:} $H$, a $\Sigma$-basis of a $\Sigma$-ideal
$I\subset P$.
\State \text{Output:} $\varphi(G)$, a \Gr\ $\Sigma$-basis of $I$
such that $\lm(G) = \lm(\varphi(G))$.
\State $G:= H^*$;
\State $B:= \{(f,g) \mid f,g\in G\}$;
\While{$B\neq\emptyset$}
\State choose $(f,g)\in B$;
\State $B:= B\setminus \{(f,g)\}$;
\ForAll{$\sigma,\tau\in\Sigma$ such that $\gcd(\sigma,\tau) = 1$}
\State $h:= \Reduce(\spoly(\sigma\cdot f,\tau\cdot g), \Sigma\cdot G\cup N)$;
\If{$h\neq 0$}
\State $h = \varphi(h)^*$
\State $B:= B\cup\{(g,h),(h,h) \mid g\in G\}$;
\State $G:= G\cup\{h\}$;
\EndIf;
\EndFor;
\EndWhile;
\State \Return $\varphi(G)$.
\end{algorithmic}
\end{algorithm}

\newpage

\begin{proposition}
The algorithm \SigmaGBasis2\ is correct.
\end{proposition}

\begin{proof}
Note that at each step we are inside an ideal $J$ such that $\varphi(J) = I$
that is whose saturation is $J' = I^*$. Moreover, for any $\ord$-homogeneous
element $h\in\bP$ which is in normal form modulo $N$ one has that $h' =
\varphi(h)^*$ divides $h$. This implies that if an S-polynomial is reduced
to zero by adding $h$ to the basis $G$, the same holds if we substitute $h$
with $h'$. In case of termination, owing to  the set $G$ is a $\ord$-homogeneous
\Gr\ $\Sigma$-basis of $J$ modulo $N$ whose elements are all saturated,
by Corollary \ref{satgb} we may conclude that $J = J'$ and hence $\varphi(G)$
is a \Gr\ $\Sigma$-basis of $I$ such that $\lm(G) = \lm(\varphi(G))$.
\end{proof}

About termination or just termination up to some order $d$, this is not
provided in general for the above algorithm. The reason is that even
if all computations are $\ord$-homogeneous, because of the saturation
$h = \varphi(h)^*$ that may decrease the order we can't be sure at some
suitable step that we will not get additional elements of order $\leq d$
in the steps that follow.


\section{An illustrative example (continued)}

We apply now the algorithm $\SigmaGBasis2$ to the same $\Sigma$-ideal that
has been considered in Section 5 for illustrating $\SigmaGBasis$.
Recall that such ideal is $I = \langle g_1,g_2 \rangle\subset P = K[X(\Sigma)]$
where $X = \{x,y\}, \Sigma = \N^2$ and
\[
\begin{array}{l}
g_1 = y(1,1)y(1,0) - 2x(0,1)^2 \\
g_2 = y(2,0) + x(0,0)x(1,0).
\end{array}
\]
Let now $\bX = \{x,y,t\}$ and define the polynomial algebra $\bP = K[\bX(\Sigma)]$
with variables $x(i,j),y(i,j),t(i,j)$, for all $i,j\geq 0$. We consider
the $\Sigma$-algebra endomorphism $\varphi:\bP \to \bP$ such that
$x(i,j)\mapsto x(i,j), y(i,j)\mapsto y(i,j)$ and $t(i,j)\mapsto 1$.
In Proposition \ref{hidbas} we proved that the largest $\ord$-graded
$\Sigma$-ideal contained in $\ker\varphi$ is the ideal $N\subset\bP$
generated by the polynomials
\[
\begin{array}{l}
t(i,j) - t(k,l)\ \mbox{where}\ (i,j)\neq (k,l), i + j = k + l; \\
t(i,j) t(k,l) - t(i,j), x(i,j) t(k,l) - x(i,j)\
\mbox{where}\ i + j\geq k + l.
\end{array}
\]
Moreover, we define a $\ord$-homogenization $\Sigma$-ordering of $\bP$
(Definition \ref{homorder}) as the weight $\Sigma$-ordering given by
the {\em degrevlex} ordering of $\Sigma$ ($\sigma_1 > \sigma_2$) and
the {\em lex} monomial ordering of $K[x(0,0),y(0,0),t(0,0)]$ ($x(0,0)\succ
y(0,0)\succ t(0,0)$). Note that such $\Sigma$-ordering extends the one
defined for $P\subset\bP$ in Section 5. In practice, it is the lexicographic
monomial ordering of $\bP$ such that
\[
\begin{array}{l}
\ldots \succ x(2, 0)\succ y(2, 0)\succ t(2, 0)\succ x(1, 1)\succ y(1, 1)
\succ t(1, 1)\succ x(0, 2)\succ \\
y(0, 2)\succ t(0, 2)\succ x(1, 0)\succ y(1, 0)\succ t(1,0)\succ x(0, 1)
\succ y(0, 1)\succ t(0, 1)\succ \\
x(0, 0)\succ y(0, 0)\succ t(0, 0).
\end{array}
\]
Recall that a \Gr\ $\Sigma$-basis of $I$ with respect to such ordering
is given by the elements $g_1,g_2$ together with
\[
\begin{array}{l}
g_3 = y(1,2)x(0,1)^2 - y(1,0)x(0,2)^2, \\
g_4 = 2x(1,1)^2 - x(0,0)x(1,0)x(0,1)x(1,1).
\end{array}
\]
We introduce then $\ord$-homogenizations of the polynomials $g_1,g_2$
that are
\[
\begin{array}{l}
g_1^* = y(1,1)y(1,0) - 2t(0,2)x(0,1)^2, \\
g_2^* = y(2,0) + t(0,2)x(0,0)x(1,0).
\end{array}
\]
We start applying $\SigmaGBasis2$ to $G = \{g_1^*,g_2^*\}$ by considering
the S-polynomials
\[ 
\begin{array}{l}
s_1 = \spoly((0,1)\cdot g_1^*, g_1^*), \\
s_2 = \spoly((1,0)\cdot g_1^*, g_2^*), \\
s_3 = \spoly((1,0)\cdot g_1^*, (0,1)\cdot g_2^*).
\end{array}
\]
By reducing $s_1$ with respect to $\Sigma\cdot G\cup N$ one obtains exactly
the polynomial
\[
g_3^* = y(1,2)x(0,1)^2 - t(0,3)y(1,0)x(0,2)^2.
\]
Clearly $g_3^*$ is already a saturated element. Then $G = \{g_1^*,g_2^*,g_3^*\}$
and we form the new S-polynomials
\[
\begin{array}{l}
s_4 = \spoly((0,1)\cdot g_1^*, g_3^*), \\
s_5 = \spoly((0,2)\cdot g_1^*, g_3^*), \\
s_6 = \spoly((0,2)\cdot g_2^*, (1,0)\cdot g_3^*).
\end{array}
\]
With respect to $\Sigma\cdot G\cup N$, one has the reductions
$s_4\to 0$ and $s_2\to h$ where
\[
h = 2t(0,3)x(1,1)^2 - t(0,3)x(0,0)x(1,0)x(0,1)x(1,1) = t(0,3)g_4^*.
\]
By saturating this element as $\varphi(h)^* = g_4^*$, we update
$G = \{g_1^*,g_2^*,g_3^*,g_4^*\}$ and another S-polynomial is defined as
\[
s_7 = \spoly((1,0)\cdot g_3^*, g_4^*).
\]
All remaining S-polynomials $s_3,s_5,s_6,s_7$ reduce to zero
with respect to $\Sigma\cdot G\cup N$ and we conclude that $\varphi(G) =
\{g_1,g_2,g_3,g_4\}$ is a \Gr\ $\Sigma$-basis of $I\subset P$.


\section{Testing and timings}

In this section we present a set of tests for the algorithms \SigmaGBasis\
and \SigmaGBasis2\ which is based on an experimental implementation of them
in the language of Maple. This is actually the first implementation
of algorithms for computing \Gr\ bases of systems of linear or
non-linear partial difference equations. Note that for the linear case
one has the packages \textsf{LDA} (Linear Difference Algebra) \cite{GR}
and \textsf{Ore\_algebra[shift\_algebra]} in the Maple distribution.
The main idea that lead us when coding the proposed algorithms is that
they can be considered variants of the classical Buchberger's algorithm
where some amount of computations can be avoided by means of the symmetry
defined by the monoid $\Sigma$. In fact, as explained in the previous
sections, a ``basic'' approach to calculate a \Gr\ $\Sigma$-basis of
a $\Sigma$-ideal $I$ generated by a $\Sigma$-basis $H$ consists in
applying the Buchberger's algorithm to the basis $\Sigma\cdot H$. One obtains
therefore a \Gr\ basis $G'$ of $I$ from which a \Gr\ $\Sigma$-basis
$G\subset G'$ can be extracted such that $\Sigma\cdot \lm(G) = \lm(G')$.
Clearly, chain and coprime criterions can be used in the usual way in the
procedure. Then, the algorithm \SigmaGBasis\ can be understood as the variant
that prescribes the application also of the $\Sigma$-criterion (Proposition
\ref{sigmacrit}) to the S-polynomials $\spoly(\sigma\cdot f,\tau\cdot g)$
and to add the set of all shifts $\Sigma\cdot h$ to the current basis
when a new element $h$ arises from the reduction of an S-polynomial.
Then, the \Gr\ $\Sigma$-basis of $I$ is simply the union of the initial
basis $H$ with the new elements $h$. Recall that the procedure is correct
only if one uses a monomial $\Sigma$-ordering. Clearly, from the set $\Sigma$
is infinite it follows that actual computations can be only performed
with a finite subset of $\Sigma$ that is over a finite set of variables
of $P = K[X(\Sigma)]$. Typically, one fixes a bound $d$ for the degree
of the elements of $\Sigma$ that is for the order of the variables
$x_i(\sigma)$. Owing to the finite $\Sigma$-criterion (Proposition
\ref{finsigmacrit}), a basis obtained with a monomial ordering compatible
with the order function is certified to be a complete \Gr\ $\Sigma$-basis
if the order bound is at least the double of the maximum top order of
its elements.

In addition to the basic procedure for the computation of \Gr\ $\Sigma$-bases
and the algorithm \SigmaGBasis, for the experiments we consider also
a variant of the latter method where the $\Sigma$-criterion is suppressed
but one continues to shift the reduced form of the S-polynomials.
This procedure is tested to the aim of understanding the contribution of
any of the implemented strategies. Finally, we propose an implementation
of the algorithm \SigmaGBasis2\ based on the saturation of a $\Sigma$-ideal
with respect to the grading defined by the order function. In practice,
once one has homogenized the initial generators, the saturation $\varphi(h)^*$ 
is performed before the application of shifting, for each new element $h$
obtained by the reduction of an S-polynomial. In output one returns the
dehomogenization of the computed basis. Note that this procedure is correct
only if one uses a $\Sigma$-ordering which is compatible with the order
function and if the polynomials are kept in normal form modulo $N$ during
the computations.

The monomial $\Sigma$-orderings of $P$ that we consider for the tests
are defined in the following way. One has initially to fix a monomial ordering
for $\Sigma$ and we choose {\em degrevlex} in order to provide compatibility
with the degree. Then, one fixes a monomial ordering, for us {\em lex},
over the subring $P(1) = K[X(1)]$ or $P(x_0) = K[x_0(\Sigma)]$ that is extended 
as a block ordering to the polynomial ring $P = K[X(\Sigma)]$ according to 
the choice of a variables ranking based on weight or index respectively.
For a detailed description of these orderings in the considered examples
see the Appendix. In the table of tests, we distinguish weight or index
ranking by the letters ``w'' and ``i''. The integer that comes before these
letters refers to the fixed order bound. Note that the algorithm \SigmaGBasis2\
is compatible only with rankings of type weight.

For the basic variant of the \Gr\ $\Sigma$-bases algorithm, one can clearly
use any implementation of the Buchberger's algorithm as, for instance,
the one contained in the package \textsf{Groebner} of Maple.
We have preferred instead to develop ourselves all different variants in order
to have the same implementation and hence the same efficiency, for the
fundamental subroutines of the algorithms. In this way, for the basic version
we have been also able to access to important parameters of the computation
as the total number of S-polynomial reductions. This number is for us the sum
of the actual S-polynomials with the initial generators that are interreduced.
Note that our implementation of the Buchberger's algorithm is in fact generally
comparable with the built-in ones of Maple. For instance, the test
{\em falkow-6w-basic} takes 9 hours, but using
\textsf{Groebner[Basis]} it takes 11 hours with \textsf{method=buchberger}
and 8.5 hours with \textsf{method=maplef4}. Other parameters
that are considered for the experiments are the number of input and output
generators. Note that for the basic algorithm we count generators and not
$\Sigma$-generators. Finally, the parameter ``minout'' refers to the number
of elements of a minimal \Gr\ $\Sigma$-basis. All examples have been computed
with Maple 12 running on a server with a four core Intel Xeon at 3.16GHz
and 64 GB RAM. The timings are given in hour-minute-second format.

\begin{center}
\begin{tabular}[t]{|l|c|c|c|c|c|}
\hline
Example           & in  & out & minout & pairs & time     \\
\hline
falkow-6w-sigma   & 4   & 5   & 5      & 5     & 18s      \\
falkow-6w-nocrit  & 4   & 5   & 5      & 8     & 57m1s    \\
falkow-6w-sigma2  & 4   & 5   & 5      & 5     & 61s      \\
falkow-6w-basic   & 157 & 157 & 5      & 157   & 9h8m11s \\
\hline
falkow-6i-sigma   & 4   & 10  & 9      & 25    & 1m45s    \\
falkow-6i-nocrit  & 4   & 10  & 9      & 34    & 1m53s    \\
falkow-6i-basic   & 157 & 163 & 9      & 172   & 1m44s    \\
\hline
navier-8w-sigma   & 4   & 6   & 5      & 9     & 26s      \\
navier-8w-nocrit  & 4   & 6   & 5      & 22    & 3h46m29s \\
navier-8w-sigma2  & 4   & 6   & 5      & 9     & 6m4s     \\
navier-8w-basic   & 86  & -   & -      & -     & $>$ 3 days \\
\hline
navier-8i-sigma   & 4   & 9   & 4      & 15    & 12s      \\
navier-8i-nocrit  & 4   & 9   & 4      & 37    & 16s      \\
navier-8i-basic   & 86  & 86  & 4      & 86    & 10s      \\
\hline
heat-12w-sigma    & 5   & 5   & 5      & 7     & 1m10s    \\
heat-12w-nocrit   & 5   & 5   & 5      & 137   & 1m46s    \\
heat-12w-sigma2   & 5   & 5   & 5      & 7     & 2m15s    \\
heat-12w-basic    & 378 & 246 & 5      & 378   & 1m33s    \\
\hline
eq26-12w-sigma    & 1   & 43  & 28     & 557   & 2m4s     \\
eq26-12w-nocrit   & 1   & 43  & 28     & 790   & 1m50s    \\
eq26-12w-sigma2   & 1   & 43  & 28     & 557   & 24m33s   \\
eq26-12w-basic    & 10  & 208 & 28     & 1673  & 6m40s    \\
\hline
eq27-12w-sigma    & 1   & 28  & 18     & 609   & 14s      \\
eq27-12w-nocrit   & 1   & 28  & 18     & 923   & 22s      \\
eq27-12w-sigma2   & 1   & 28  & 18     & 609   & 25s      \\
eq27-12w-basic    & 9   & 121 & 18     & 726   & 11s      \\
\hline
\end{tabular}
\end{center}

\bigskip
We give now some informations about the examples we have used for the experiments.
See the Appendix of the paper for an explicit description of the test set.
All considered examples are non-linear systems of ordinary or partial difference
equations with constant coefficients which are of interest in literature.
For instance, the tests {\em falkow} are obtained by the discretization of
the Falkowich-Karman equation which is a non-linear two-dimensional differential
equation describing transonic flow in gas dynamics. The discretization
we used are equations (41) in \cite{GBM}. Then, the {\em navier} examples
are based on equations $e_1,e_2,e_3,e_4$ of the system (13) in the paper
\cite{GB} that are a finite difference scheme corresponding to the discretization
(9) of the Navier-Stokes equations for two-dimensional viscous incompressible
fluid flows. The tests {\em heat} are the discretization of the
one-dimensional heat equation as described in the equations (10) and (11)
of \cite{Lv}. Finally, {\em eq26} and {\em eq27} are the equations (2.6)
and (2.7) at page 24 of \cite{GL} which are examples of ordinary difference
equations that have periodic solutions.

By analyzing the experiments, it is sufficiently clear that the strategy
implemented in \SigmaGBasis\ is the safest one and hence on the average,
the most efficient one. In fact, by decreasing the number of S-polynomials
this strategy avoids the dramatical effects of involved reductions as for
the tests {\em falkow-6w} and {\em navier-8w}. For simpler examples the
four strategies appear essentially equivalent. The algorithms \SigmaGBasis\
and \SigmaGBasis2\ lead to practically identical computations but the latter
method suffers of some overhead which is probably due to our still experimental
implementation. For instance, even if the normal form modulo the ideal $N$
is described in the Definition \ref{Nnform}, in our implementation
we obtain it by computations that is by adding a \Gr\ basis of $N$ to
the input basis for \SigmaGBasis2.

The proposed algorithms usually provide only partial informations about
the structure of \Gr\ $\Sigma$-bases since they are in general infinite.
Nevertheless, it is interesting to note that by means of the finite
$\Sigma$-criterion we have been able to certify that the examples
{\em falkow, navier} and {\em heat} have finite bases with respect to the
weight ranking. In particular, the elements of the \Gr\ basis of {\em falkow}
have maximum top order equal to 4 and hence they are certified in order 8
in about 4 minutes. The example {\em navier} has max top order equal
to 6 and its certification is obtained in order 12 in less than one hour.
Finally, the example {\em heat} has max top order 2 and it gets
certification in order 4 in 0 seconds.


\section{Conclusions and future directions}

This paper shows that one can not only generalize in a systematic way
the \Gr\ bases theory and related algorithms to the algebras of partial
difference polynomials but also make these methods really work by
introducing suitable gradings for such algebras. In fact, weight and order
functions provide a Noetherian subalgebras filtration that implies
termination and completeness certification for actual computations
that are performed within some bounded degree that is over a finite
number of variables. We have then developed the first implementation
of a variant of the Buchberger's algorithm for systems of linear or
non-linear partial difference equations. Even if such implementation
is just experimental, the approach corresponding to the algorithm
\SigmaGBasis\ is strong enough to let it able to work with
discretizations of real world systems of non-linear differential
equations.

In this paper we consider difference equations with constant
coefficients and hence a next step along this line of research
is to extend the proposed methods to systems of difference equations
with non-constant coefficients that is to assume that $\Sigma$ acts
on the base field $K$ in a non-trivial way. Moreover, since the algebras
of partial difference polynomials are free objects in the category
of commutative algebras endowed with the action of a monoid $\Sigma$
isomorphic to $\N^r$, a natural future direction consists in extending
the ideas introduced here to other types of monoid symmetry over
commutative algebras as the ones used, for instance, in algebraic statistic
\cite{BD,HM}. Starting from \Gr\ bases, classical directions are the computation
of the Hilbert series and free resolutions that one may generalize
to partial difference ideals or other types of invariant ideals.
Finally, we aim to have the proposed algorithms implemented in the kernel
of computer algebra systems in order to tackle involved problems related
with the discretization of systems of partial differential equations
\cite{Ge,GBM,GR}.


\section*{Acknowledgments}

The author would like to thank Vladimir Gerdt for introducing him
to the theory of difference algebras and supporting the preparation
of testing examples. He is also grateful to the research group of
\textsf{Singular} \cite{DGPS} for the courtesy of allowing access
to their servers for performing computational experiments.
Thanks also to the reviewer for all valuable remarks that have helped
to make the paper more readable.

\section*{Appendix: the test set}

\noindent
{\bf Falkovich-Karman} ({\em falkow})

\medskip \noindent
The base field is $F = \Q(h,\tau,K,\gamma)$ (field of rational functions with
rational coefficients) where $h$ is the space mesh step, $\tau$ is the time
mesh step and $K,\gamma$ are parameters of the corresponding differential
equation. The algebra of partial difference polynomials is $P =
F[\varphi_x(i,j,k),\varphi_y(i,j,k),\varphi_t(i,j,k),\varphi(i,j,k)\mid i,j,k\geq 0]$.
The monomial ordering of $P$ is the lexicographic one based on the following
weight ranking
\[
\begin{array}{l}
\ldots \succ
\varphi_x(2,0,0)\succ \varphi_y(2,0,0)\succ \varphi_t(2,0,0)\succ \varphi(2,0,0)\succ
\varphi_x(1,1,0)\succ \\
\varphi_y(1,1,0)\succ \varphi_t(1,1,0)\succ \varphi(1,1,0)\succ \varphi_x(0,2,0)\succ
\varphi_y(0,2,0)\succ \varphi_t(0,2,0)\succ \\
\varphi(0,2,0)\succ \varphi_x(1,0,1)\succ \varphi_y(1,0,1)\succ \varphi_t(1,0,1)\succ
\varphi(1,0,1)\succ \varphi_x(0,1,1)\succ \\
\varphi_y(0,1,1)\succ \varphi_t(0,1,1)\succ \varphi(0,1,1)\succ \varphi_x(0,0,2)\succ
\varphi_y(0,0,2)\succ \varphi_t(0,0,2)\succ \\
\varphi(0,0,2)\succ \varphi_x(1,0,0)\succ \varphi_y(1,0,0)\succ \varphi_t(1,0,0)\succ
\varphi(1,0,0)\succ \varphi_x(0,1,0)\succ \\
\varphi_y(0,1,0)\succ \varphi_t(0,1,0)\succ \varphi(0,1,0)\succ \varphi_x(0,0,1)\succ
\varphi_y(0,0,1)\succ \varphi_t(0,0,1)\succ \\
\varphi(0,0,1)\succ \varphi_x(0,0,0)\succ \varphi_y(0,0,0)\succ \varphi_t(0,0,0)\succ
\varphi(0,0,0),
\end{array}
\]
or on the following index ranking
\[
\begin{array}{l}
\ldots \succ 
\varphi_x(2,0,0)\succ \varphi_x(1,1,0)\succ \varphi_x(0,2,0)\succ
\varphi_x(1,0,1)\succ \varphi_x(0,1,1)\succ \\
\varphi_x(0,0,2)\succ \varphi_x(1,0,0)\succ \varphi_x(0,1,0)\succ
\varphi_x(0,0,1)\succ \varphi_x(0,0,0)\succ \\
\ldots \succ 
\varphi_y(2,0,0)\succ \varphi_y(1,1,0)\succ \varphi_y(0,2,0)\succ
\varphi_y(1,0,1)\succ \varphi_y(0,1,1)\succ \\
\varphi_y(0,0,2)\succ \varphi_y(1,0,0)\succ \varphi_y(0,1,0)\succ
\varphi_y(0,0,1)\succ \varphi_y(0,0,0)\succ \\
\ldots \succ 
\varphi_t(2,0,0)\succ \varphi_t(1,1,0)\succ \varphi_t(0,2,0)\succ
\varphi_t(1,0,1)\succ \varphi_t(0,1,1)\succ \\
\varphi_t(0,0,2)\succ \varphi_t(1,0,0)\succ \varphi_t(0,1,0)\succ
\varphi_t(0,0,1)\succ \varphi_t(0,0,0)\succ \\
\ldots \succ 
\varphi(2,0,0)\succ \varphi(1,1,0)\succ \varphi(0,2,0)\succ
\varphi(1,0,1)\succ \varphi(0,1,1)\succ \\
\varphi(0,0,2)\succ \varphi(1,0,0)\succ \varphi(0,1,0)\succ
\varphi(0,0,1)\succ \varphi(0,0,0).
\end{array}
\]
The input partial difference polynomials for the computational experiments are
\[
\begin{array}{l}
E_1 = - 4 h p(2,1,1) - h \tau(\gamma + 1) p_x(2,1,0)^2 +
2 h \tau K p_x(2,1,0) + 4 h p(2,1,0) + \\
2 h \tau p_y(1,2,0) - 4 h^2 p_t(1,1,0) + 4 h p(0,1,1) -
2 h \tau p_y(1,0,0) + \\
h \tau (\gamma + 1) p_x(0,1,0)^2 - 2 h \tau K p_x(0,1,0) - 4 h p(0,1,0), \\
E_2 = \frac{h}{2} (p_x(1,0,0) + p_x(0,0,0)) - p(1,0,0) + p(0,0,0), \\
E_3 = \frac{h}{2} (p_y(0,1,0) + p_y(0,0,0)) - p(0,1,0) + p(0,0,0), \\
E_4 = 2 \tau p_t(0,0,1) - p(0,0,2) + p(0,0,0).
\end{array}
\]

\medskip \noindent
{\bf Navier-Stokes} ({\em navier})

\medskip \noindent
The base field is $K = \Q(h,\tau,Re)$ where $Re$ is the Reynolds number.
The algebra of partial difference polynomials is $P = K[p(i,j,k),u(i,j,k),v(i,j,k)\mid
i,j,k\geq 0]$. The lexicographic monomial ordering of $P$ is defined by
following weight ranking
\[
\begin{array}{l}
\ldots \succ p(2, 0, 0)\succ u(2, 0, 0)\succ v(2, 0, 0)\succ
p(1, 1, 0)\succ u(1, 1, 0)\succ v(1, 1, 0)\succ \\
p(0, 2, 0)\succ u(0, 2, 0)\succ v(0, 2, 0)\succ p(1, 0, 1)\succ
u(1, 0, 1)\succ v(1, 0, 1)\succ p(0, 1, 1)\succ \\
u(0, 1, 1)\succ v(0, 1, 1)\succ p(0, 0, 2)\succ u(0, 0, 2)\succ
v(0, 0, 2)\succ p(1, 0, 0)\succ u(1, 0, 0)\succ \\
v(1, 0, 0)\succ p(0, 1, 0)\succ u(0, 1, 0)\succ v(0, 1, 0)\succ
p(0, 0, 1)\succ u(0, 0, 1)\succ v(0, 0, 1)\succ \\
p(0, 0, 0)\succ u(0, 0, 0)\succ v(0, 0, 0),
\end{array}
\]
or by the following index ranking
\[
\begin{array}{l}
\ldots \succ p(2, 0, 0)\succ p(1, 1, 0)\succ p(0, 2, 0)\succ
p(1, 0, 1)\succ p(0, 1, 1)\succ p(0, 0, 2)\succ \\
p(1, 0, 0)\succ p(0, 1, 0)\succ p(0, 0, 1)\succ p(0, 0, 0)\succ
\ldots \succ u(2, 0, 0)\succ u(1, 1, 0)\succ \\
u(0, 2, 0)\succ u(1, 0, 1)\succ u(0, 1, 1)\succ u(0, 0, 2)\succ
u(1, 0, 0)\succ u(0, 1, 0)\succ \\
u(0, 0, 1)\succ u(0, 0, 0)\succ \ldots \succ v(2, 0, 0)\succ v(1, 1, 0)\succ
v(0, 2, 0)\succ v(1, 0, 1)\succ \\
v(0, 1, 1)\succ v(0, 0, 2)\succ v(1, 0, 0)\succ v(0, 1, 0)\succ
v(0, 0, 1)\succ v(0, 0, 0).
\end{array}
\]
The input partial difference polynomials are
\[
\begin{array}{l}
E_1 = 2 h (u(0,2,1) - u(0,0,1) + v(0,1,2) - v(0,1,0)), \\
E_2 = - 4 \frac{\tau}{Re} h^2 (u(0,4,2) + u(0,2,4)) + 16 h^4 u(1,2,2) +
8 \tau h^3 (u(0,3,2)^2 + \\
u(0,2,3) v(0,2,3) + p(0,3,2)) + 16 h^2 (- h^2 + \frac{\tau}{Re}) u(0,2,2) -
8 \tau h^3 (u(0,1,2)^2 + \\
u(0,2,1) v(0,2,1) + p(0,1,2)) - 4 \frac{\tau}{Re} h^2 (u(0,0,2) + u(0,2,0)), \\
E_3 = - 4 \frac{\tau}{Re} h^2 (v(0,4,2) + v(0,2,4)) + 16 h^4 v(1,2,2) +
8 \tau h^3 (v(0,2,3)^2 + \\
u(0,3,2) v(0,3,2) + p(0,2,3)) + 16 h^2 (- h^2 + \frac{\tau}{Re}) v(0,2,2) -
8 \tau h^3 (v(0,2,1)^2 + \\
u(0,1,2) v(0,1,2) + p(0,2,1)) - 4 \frac{\tau}{Re} h^2 (v(0,2,0) + v(0,0,2)), \\
E_4 = 4 h^2 (v(0,2,4)^2 + u(0,4,2)^2 + v(0,2,0)^2 + u(0,0,2)^2
+ p(0,4,2) + \\
p(0,2,4) + p(0,2,0) + p(0,0,2)) + 8 h^2 (u(0,3,3) v(0,3,3) - v(0,2,2)^2 - \\
u(0,2,2)^2 - u(0,3,1) v(0,3,1) - u(0,1,3) v(0,1,3) + u(0,1,1) v(0,1,1)) - \\
16 h^2 p(0,2,2).
\end{array}
\]

\medskip \noindent
{\bf Heat} ({\em heat})

\medskip \noindent
The base field is by definition $K = \Q(h,\tau)$ and the partial difference polynomial
algebra is $P = K[x(i,j),t(i,j),u(i,j)\mid i,j\geq 0]$. The lexicographic monomial
order of $P$ is obtained by the following weight ranking
\[
\begin{array}{l}
\ldots \succ x(2, 0)\succ t(2, 0)\succ u(2, 0)\succ x(1, 1)\succ
t(1, 1)\succ u(1, 1)\succ x(0, 2)\succ \\
t(0, 2)\succ u(0, 2)\succ x(1, 0)\succ t(1, 0)\succ u(1, 0)\succ
x(0, 1)\succ t(0, 1)\succ u(0, 1)\succ \\
x(0, 0)\succ t(0, 0)\succ u(0, 0),
\end{array}
\]
The input is given by
\[
\begin{array}{l}
E_1 = (u(1,0) - u(0,0)) (x(0,1) - x(0,0))^2 -
(u(0,2) - 2 u(0,1) + \\
u(0,0)) (t(1,0) - t(0,0)), E_2 = x(1,0) - x(0,0), E_3 = x(0,1) - x(0,0) - h, \\
E_4 = t(1,0) - t(0,0) - \tau,
E_5 = t(0,1) - t(0,0).
\end{array}
\]

\medskip \noindent
{\bf Equation 2.6} ({\em eq26})

\medskip \noindent
One considers the algebra of ordinary difference polynomials $P = \Q[x(i)\mid i\geq 0]$
endowed with the lexicographic monomial ordering such that
\[
\ldots \succ x(4)\succ x(3)\succ x(2)\succ x(1)\succ x(0).
\]
The input ordinary difference polynomial is
\[
\begin{array}{l}
E = x(3)x(0) - x(2) - x(1) - 1.
\end{array}
\]

\medskip \noindent
{\bf Equation 2.7} ({\em eq27})

\medskip \noindent
We consider the same algebra $P = \Q[x(i)\mid i\geq 0]$ endowed with
the same monomial ordering as in the previous example. The input polynomial is here
\[
\begin{array}{l}
E = x(4)x(2)x(0) - x(3)x(1).
\end{array}
\]


\end{document}